\documentclass[reqno]{amsart}
\usepackage[scale=0.75, centering, headheight=14pt]{geometry}
\usepackage[latin1]{inputenc}
\usepackage[T1]{fontenc}
\usepackage{lmodern}
\usepackage[english]{babel}
\usepackage{microtype}
\usepackage[fontsize=11pt]{scrextend}
\usepackage{amsmath,amssymb,amsfonts,amsthm}
\usepackage{mathtools,accents}
\usepackage{mathrsfs}
\usepackage{aliascnt}
\usepackage{braket}
\usepackage{bm}
\usepackage[initials]{amsrefs}
\usepackage[citecolor=blue,colorlinks]{hyperref}
\addto\extrasenglish{}

\usepackage{enumerate}
\usepackage{xcolor}
\usepackage{soul}
\usepackage{aliascnt}

\makeatletter
\def\newaliasedtheorem#1[#2]#3{
  \newaliascnt{#1@alt}{#2}
  \newtheorem{#1}[#1@alt]{#3}
  \expandafter\newcommand\csname #1@altname\endcsname{#3}
}
\makeatother

\theoremstyle{plain}
\newtheorem{theorem}{Theorem}[section]
\newaliasedtheorem{lemma}[theorem]{Lemma}
\newaliasedtheorem{prop}[theorem]{Proposition}
\newaliasedtheorem{claim}[theorem]{Claim}
\newaliasedtheorem{corollary}[theorem]{Corollary}

\theoremstyle{definition}
\newaliasedtheorem{definition}[theorem]{Definition}
\newaliasedtheorem{example}[theorem]{Example}

\theoremstyle{remark}
\newaliasedtheorem{remark}[theorem]{Remark}

\numberwithin{equation}{section}

\def\R{\mathbb R}
\def\C{{\mathbb C}}
\def\N{{\mathbb N}}

\DeclareMathOperator*{\esssup}{ess\,sup}

\title[Dynamical Collapse of dipolar BEC]{Dynamical collapse of cylindrical symmetric dipolar Bose-Einstein condensates} 

\author[J. Bellazzini \and L. Forcella]{Jacopo Bellazzini \and Luigi Forcella}

\address[J. Bellazzini]{Dipartimento di Matematica, Universit\`a Degli Studi di Pisa, Largo Bruno Pontecorvo, 5, 56127, Pisa, Italy}
\email{jacopo.bellazzini@unipi.it}

\address{Luigi Forcella\hfill\break  \'Ecole Polytechnique F\'ed\'erale de Lausanne, Institute of Mathematics, Station 8, CH-1015 Lausanne, Switzerland.}
\email{luigi.forcella@epfl.ch}



\begin{document}

\maketitle


\begin{abstract}
We study the formation of singularities for  cylindrical  symmetric solutions to the Gross-Pitaevskii equation describing a dipolar Bose-Einstein condensate. We prove  that solutions arising from initial data with energy below the energy of the Ground State and that do not scatter collapse in finite time. The main tools to prove our result are the variational characterization of the Ground State energy, suitable localized virial identities for cylindrical symmetric functions, and  general integral and pointwise estimates for operators involving powers of the Riesz transform.
\end{abstract}

\section{Introduction}

Since the first experimental observation in $1995$  of a quantum state of matter at very low temperature called Bose-Einstein condensate (BEC), see e.g. \cite{AEMWC,BrSaToHu, DMAVDKK }, the study of the asymptotic dynamics of nonlinear equations describing this phenomena rapidly increased, both numerically and theoretically. Since BEC exists in an ultracold and
dilute regime, the most relevant interactions are the isotropic, elastic two-body collisions. After the first pioneering experimental works, other condensates have
been produced with different atoms, in particular condensates made out of particles possessing a permanent electric or magnetic dipole moment.
Such kind of condensates are called dipolar Bose-Einstein condensates, see e.g. \cite{ BaCa, BaCaWa, NaPeSa, PS,  SSZL}, and  their peculiarity  is given by the  long-range anisotropic interaction between particles, in contrast with the short-range, isotropic character of the contact
interaction of BEC.  \\
A dipolar quantum gases is well modelled, see \cite{YY1,YY2,LMS} for the validity of such model, by the Gross-Pitaevskii equation (GPE)
\begin{equation}
\label{eq:evolution}
i h \frac{\partial u}{\partial t} = - \frac{h^2}{2m}\Delta u + W(x) u + U_0|u|^2 u + (V_{dip}\ast |u|^2) u, 
\end{equation}
where the wave function $u=u(t,x).$  Here $t$ is the time variable, $x = (x_1,x_2,x_3)$ is the space coordinate, $h$ is the Planck constant, $m$ is the mass of a dipolar particle and $W(x)$ is an external, real potential which describes the electromagnetic trap. The coefficient $U_0 = 4 \pi h^2 a_s /m$ describes the local interaction between dipoles in the condensate, $a_s$ being the $s$-wave scattering length (positive for repulsive interactions and negative for attractive interactions).
The long-range dipolar interaction potential between two dipoles is given by
\begin{equation*}
V_{dip}(x) = \frac{\mu_0 \mu^2_{dip}}{4 \pi} \,  \frac{1 - 3\cos^2 (\theta)}{|x|^3}, \quad x \in \R^3,
\end{equation*}
where $\mu_0$ is the vacuum magnetic permeability, $\mu_{dip}$ is the permanent magnetic dipole moment and $\theta$ is the angle between the dipole axis $n\in\R^3$ and the vector $x$.  For simplicity, we fix the dipole axis as the vector $n=(0,0,1)$.
The wave function is normalized according to
\begin{equation*}
\int_{\R^3} |u(x,t)|^2\,dx = N,
\end{equation*}
where $N$ is the total number of dipolar particles in the dipolar BEC.  In this work we consider the case when the  trapping potential $W$ is not active, i.e. we freeze $W(x)=0.$\\

As we are interested in the mathematical features of the GPE,  we consider it in its dimensionless form, therefore we write the Cauchy problem associated to \eqref{eq:evolution}  as follows:
\begin{equation}\label{GP}
\left\{ \begin{aligned}
i\partial_{t}u+\frac12\Delta u&=\lambda_1|u|^{2}u+\lambda_2(K\ast|u|^2)u, \quad (t,x)\in \R\times 
\mathbb{R}^3\\
u(0,x)&=u_0(x)
\end{aligned}\right.,
\end{equation}
where  the dipolar kernel $K$ acting in the convolution on the mass density $|u|^2$ is given by 
\begin{equation*}
K(x)=\frac{x_1^2+x_2^2-2x_3^2}{|x|^5}.
\end{equation*}
The two coefficients $\lambda_{1,2}$  involved in the equation are two physical, real parameters defined by
\begin{equation*}
\lambda_1 = 4 \pi a_s N \gamma, \quad \lambda_2 = \frac{mN \mu_0 \mu_{dip}^2 }{4 \pi h^2}\gamma;
\end{equation*}
they describe the strength of the local nonlinearity and the nonlocal nonlinearity, respectively.
Following the terminology introduced by  Carles, Markowich and Sparber in  \cite{CMS}, where the authors give a first mathematical treatment concerning various aspects about local/global well-posedness of \eqref{GP}, 
we consider the partition of the coordinate plane $(\lambda_1,\lambda_2)$ into   the so-called  \emph{Unstable Regime}
\begin{equation}\label{UR}
\left\{ \begin{aligned}
\lambda_1-\frac{4\pi}{3}\lambda_2<0 &\quad \hbox{ if } \quad \lambda_2>0\\
\lambda_1+\frac{8\pi}{3}\lambda_2<0 & \quad\hbox{ if } \quad\lambda_2<0
\end{aligned}\right.,
\end{equation}
and its complementary, the so-called  \emph{Stable Regime}
\begin{equation}\label{SReg}
\left\{ \begin{aligned}
\lambda_1-\frac{4\pi}{3}\lambda_2\geq 0 &\quad \hbox{ if } \quad \lambda_2>0\\
\lambda_1+\frac{8\pi}{3}\lambda_2\geq 0 & \quad\hbox{ if } \quad\lambda_2<0
\end{aligned}\right..
\end{equation}

Heuristically, when comparing \eqref{GP} to the classical cubic NLS (i.e. when $\lambda_2=0$), one can think to the configuration given by \eqref{UR} as the nonlinearity were focusing, and to \eqref{SReg} as the nonlinearity were defocusing. This notation is although incorrect in the context of the GPE, as we will emphasize in some remark later on in the paper, after we introduce some basic notation. 

Solutions to \eqref{GP} conserve along the flow the mass and the energy (besides other quantities not used in this paper); more rigorously  
\begin{equation*}
M(t)=M(u(t)):=\int_{\R^3}|u(t)|^2\,dx=M(0)
\end{equation*}
and
\begin{equation}\label{energy}
E(t)=E(u(t)):=\frac{1}{2}\left(\int_{\R^3}|\nabla u(t)|^2+\lambda_1|u(t)|^4+\lambda_2(K\ast|u(t)|^2)|u(t)|^2\,dx\right)=E(0),
\end{equation} 
for any $t\in(-T_{min},T_{max}),$ where $T_{min},T_{max}\in(0,\infty]$ are the minimal and maximal time of existence of the solution, respectively.  Local existence of solutions to \eqref{GP} was shown in \cite{CMS}, in both the configurations given by \eqref{UR} and \eqref{SReg}.\\

The Unstable Regime \eqref{UR} is of particular relevance, since stationary solutions are allowed in this region. More precisely, we recall that stationary states  are solutions of the following species:
\begin{equation*}
u(x,t) = e^{-i \kappa  t}u(x), 
\end{equation*}
where $u(x)$ is a time-independent function solving the stationary equation
\begin{equation}
\label{eq:maina}
- \frac{1}{2}\Delta u + \lambda_1 |u|^2 u + \lambda_2 (K \ast  |u|^2) u + \kappa u =0
\end{equation}
constrained on the manifold $S(1),$ where  
\begin{equation}\label{constraint1}
S(1) = \{ u \in H^1(\R^3) \ s.t. \ \|u\|_{L^2(\R^3)}^2 =1\},
\end{equation}
and $\kappa \in \R$ is a real parameter usually referred as the chemical potential. We postpone the rigorous discussion about existence of solutions to \eqref{eq:maina} in  \autoref{sec:3}. We introduce now some crucial quantities often used along the paper, and we proceed enunciating the main results of this work and the strategy to get them.
\\

Let us recall some notation consistent  to our previous papers \cite{BJ, BF19}: by means of the Plancherel identity, the energy defined in \eqref{energy} can be rewritten as
\begin{equation*}
E(t)=\frac{1}{2}\int_{\R^3}|\nabla u(t)|^2\,dx+\frac{1}{2(2\pi)^3}\int_{\R^3}\left(\lambda_1+\lambda_2\hat K(\xi)\right)(\widehat{|u(t)|^2} )^2(\xi)\,d\xi
\end{equation*}
where  the Fourier transform of  the dipolar kernel $K$ is explicitly given by
\begin{equation}\label{kernel:fou}
\hat K(\xi)=\frac{4\pi}{3}\frac{2\xi_3^2-\xi_2^2-\xi_1^2}{|\xi|^2}, \qquad \xi\in\R^3.
\end{equation}
We refer to  \cite{CMS} for a proof of the explicit calculation of $\hat K,$ done by means of the decomposition in spherical harmonics of the Fourier character $e^{-ix\cdot\xi}.$ 
\begin{remark}\label{rem:k-rie}
It is worth mentioning right now that, by using  \eqref{kernel:fou}, $\hat K(\xi)$  is a linear combination of  symbols  associated to the square of the Riesz transforms $\mathcal R_j^2$ for $j=1,2,3,$ i.e. $\widehat{\mathcal R_j^2f}(\xi)=-\frac{\xi_j^2}{|\xi|^2}\hat f(\xi).$ Therefore $(K\ast f)(x)$ is a linear combination of $\mathcal R_j^2f(x)$'s.
\end{remark}
A trivial computation provides a lower and an upper bound for $\hat K,$ and more precisely
\begin{equation*}
\hat K\in\left[-\frac43\pi,\frac83\pi\right],
\end{equation*} 
and from the latter it is straightforward to claim that the convolution with $K$ defines an $L^2\mapsto L^2$ continuous operator. 
 
We split the energy as sum of the kinetic and potential energies, respectively defined by 
\begin{equation}\label{kinetic:en}
T(u)=\int_{\R^3}|\nabla u|^2\,dx
\end{equation} 
and
\begin{equation}\label{potential:en}
P(u)=\frac{1}{(2\pi)^3}\int_{\R^3}\left(\lambda_1+\lambda_2\hat K(\xi)\right)(\widehat{|u|^2})^2(\xi)\,d\xi,
\end{equation}
\noindent and we introduce the quantity 
\begin{equation}\label{GG}
G(u)=T(u)+\frac32P(u).
\end{equation}
Moreover, the following useful identity holds true: $E-\frac{1}{3}G=\frac 16 T.$ The functional $G$ naturally appears by means of the Pohozaev identities related to \eqref{eq:maina}, see \cite{AS}.\\

In spite of  the fact that we are primarily interested in solutions satisfying \eqref{constraint1}, we consider, for a positive $c>0,$ the generic manifold
\begin{equation*}
S(c)=\left\{ u \in H^1(\R^3) \ s.t. \ \|u\|_{L^2(\R^3)}^2=c\right\},
\end{equation*}
which will be useful for the mathematical study of existence of standing states, and their variational characterization.  The case $c=1$ clearly corresponds to the mass normalization  expressed in \eqref{constraint1}.
For a fixed $c>0$, the energy $E(u)$ has a mountain pass geometry on $S(c)$ and we denote by $\gamma(c)$ the mountain pass energy at level $c$ to which it corresponds a stationary  state. Again, we refer to \autoref{sec:3} for precise definitions and rigorous results.  

The energy level $\gamma(c)$  has the variational characterization below, that will be essential in the sequel; by introducing the manifold
\[
V(c)=\{u\in H^1(\R^3) \quad s.t. \quad \|u\|_{L^2(\R^3)}^2=c  \hbox{ and } G(u)=0\}
\]
we recall, see \cite{BJ}, that 
\begin{equation}\label{def:vc}
\gamma(c)=\inf\{E(u) \quad s.t. \quad u\in V(c)\}.
\end{equation} 

A contradiction argument in conjunction with a continuity argument implies that provided  $E(u_0)<\gamma(c),$ with $c=\|u_0\|_{L^2(\R^3)}^2$ and $G(u_0)>0,$  the local solution $u\in\mathcal C((-T_{min},T_{max});H^1(\R^3))$ to \eqref{GP} can be extended globally in time, i.e. $T_{min}=T_{max}=\infty,$ and $G(u(t))>0$ for any $t\in\R,$ see \cite[Theorem 1.3]{BJ}.\\

The global existence of solutions under the hypothesis $E(u_0)<\gamma(c),$ with $c=\|u_0\|_{L^2(\R^3)}^2,$ and $G(u_0)>0,$ suggests the possibility that all solutions arising from these initial data scatter, in analogy of what was proved by Duyckaerts, Holmer, and Roudenko in \cite{HR, DHR} for the cubic focusing NLS by exploiting the original approach of concentration/compactness and rigidity method in the spirit of  the Kenig and Merle road map, see \cite{KM1}. The authors in fact recently proved  in \cite{BF19} that the conditions  $E(u_0)<\gamma(c),$ for $c=\|u_0\|_{L^2(\R^3)}^2,$  and  $G(u_0)>0$ imply scattering of solutions to the dipolar Gross-Pitaevskii equation \eqref{GP}. 

The main aim of this paper is to study the asymptotic dynamics in the complementary configuration, i.e.  $E(u_0)<\gamma(c),$ with $c=\|u_0\|_{L^2(\R^3)}^2,$ and $G(u_0)<0.$  We shall underline that even in the case when  $\lambda_2=0$ and $\lambda_1<0$, namely when  \eqref{GP} reduces to the classical focusing cubic nonlinear Schr\"odinger equation 
\begin{equation}\label{NLS}
\left\{ \begin{aligned}
i\partial_{t}u+\frac12\Delta u&=\lambda_1|u|^{2}u, \quad (t,x)\in \R\times 
\mathbb{R}^3\\
u(0,x)&=u_0(x)
\end{aligned}\right.,
\end{equation}
finite time blow-up for  any initial data $u_0\in H^1(\mathbb R^3)$ satisfying the above conditions is still an open problem. To the best of our knowledge, the less restrictive  assumptions in this context are due to Martel, see \cite{Mar}, where the author proves finite time blow-up in  the space of cylindrical symmetric functions with finite variance in the $x_3$ direction and negative energy. Early results of this type are due to Glassey, see \cite{Gla}, in case of finite variance, and Ogawa and Tsutsumi \cite{OT}, in the radial symmetric case (see also Holmer and Roudenko \cite{HR}, and Inui \cite{Inui1, Inui 2} in a more general setting).\\

We are now in position to state our main results and to explain our strategy of the proofs.  Let us define $\bar x=(x_1 ,x_2 )$ and let us introduce the space where we study the formation of singularities: 
\begin{equation*}
\Sigma_3 =\left\{ u \in H^1(\R^3)\quad s.t. \quad u(x)=u(|\bar x|, x_3) \ \hbox{ and } \ x_3 u\in L^2(\R^3) \right\}.
\end{equation*}
$\Sigma_3$ is therefore the space of cylindrical symmetric functions with finite variance in the $x_3$ direction. \\

Our main result is as follows. 
\begin{theorem}\label{thm:main}
Assume that $\lambda_1, \lambda_2$ satisfy \eqref{UR}, namely the belong to the Unstable Regime. Let $u(t)\in \Sigma_3$ be a solution to \eqref{GP} defined on $(-T_{min}, T_{max}),$ with initial datum $u_0$ satisfying $E(u_0)<\gamma(\|u_0\|_{L^2}^2)$ and $G(u_0)<0.$ Then $T_{min}$ and $T_{max}$ are finite, namely $u(t)$ blows-up in finite time. 
\end{theorem}

As a consequence of \autoref{thm:main} we give  a generalization of the result by  Martel in \cite{Mar}. We extend here  that result for all positive initial energies under the energy threshold given by the Ground State associated to NLS (which corresponds to the one given in \eqref{eq:maina} for $\lambda_2=0$). Even if the following theorem can be viewed as a straightforward corollary of  \autoref{thm:main}, we prefer to state it as an independent result, since it has its own interest.

\begin{theorem}\label{thm:main-nls}
Given a solution $u(t)\in \Sigma_3$ to \eqref{NLS} with $\lambda_1<0$ defined on $(-T_{min}, T_{max}),$ with initial datum $u_0$ satisfying $E(u_0)<\gamma(\|u_0\|_{L^2}^2)$ and $G(u_0)<0,$ then $T_{min}$ and $T_{max}$ are finite, namely $u(t)$ blows-up in finite time. 
\end{theorem}

We point out some feature of the dipolar GPE.

\begin{remark}
Blow-up in finite time for the focusing cubic NLS in the whole generality, i.e. for infinite-variance initial data and without assuming any symmetry, is still an open problem. See the beginning of \autoref{sec:NLS} for up-to-date references. 
\end{remark}

\begin{remark}
A usual assumption that often appears in literature in order to simplify the analysis of a model which cannot be treated in a full generality,  is the restriction to a radial setting.  The dipolar kernel $K(x)$ is a Calder\'on-Zigmund operator  of the form $|x|^{-3}\mathcal O(x)$ where $\mathcal O$ is a  zero-order function having zero average on the sphere. This implies that the restriction to radial symmetric solutions makes disappear the effect of the nonlocal term in \eqref{GP}, hence the equation reduces to the classical NLS equation in the radial framework, see \cite{CMS}.
\end{remark}

\begin{remark}
We restrict  the functional space to  functions belonging to $H^1(\R^3)$ with finite variance in the $x_3$ direction. As it will be clear along the  proof a decay estimate for the potential energy in the exterior of a cylindrical domain will be crucial. Indeed, the finite variance in the $x_3$ direction enables us to localize the potential energy on the exterior of a cylinder.
\end{remark}
\begin{remark}
As remarked in \cite{CMS}, for $0<\lambda_1<\frac{4\pi}{3}\lambda_2$ -- namely when the local nonlinearity is defocusing, and the coupling parameter $\lambda_2$ is positive, which is the physical case -- finite time blow-up may arise, so  that is improper  to speak about ``defocusing/focusing'' for the dipolar BEC. The nonlocal interaction then can yield to formation of singularities in finite time of the solutions. 
\end{remark}
\begin{remark}
From the identity \eqref{GG} and the fact that there exists a positive constant $\delta>0$ such that $G(u(t))\leq-\delta$ for any $t\in(-T_{min}, T_{max})$ (see \autoref{lem:2.2} below), it is straightforward to see that the assumption $G(u_0)<0$ implies that $P(u(t))<0$ for any time in the maximal interval of existence of the solution to \eqref{GP}. 
This is in contrast of what happens in the counterpart scenario $G(u_0)>0.$ In the latter case, working  in the Unstable Regime \eqref{UR} does not guarantee that the potential energy $P(u(t))$ preserves the sign, as we proved in \cite{BF19}.
\end{remark}

We turn now to state the ingredients we use in order to prove our main theorems. The strategy and the main difficulties are the following.
\begin{itemize} 
\item A variational characterization of the Ground State energy which firstly permits to  prove that $\|u(t)\|_{\dot H^1(\R^3)}$ is bounded from below uniformly in time  and that  $G(u(t))<-\delta$ for all times in the maximal interval of existence of the solution.  As a byproduct,  which is crucial for what follows, it exploits the bound $G(u(t))\leq -\tilde\delta \|u(t)\|_{\dot H^1(\R^3)}^2,$ for some $\tilde\delta>0.$
\item The virial identities, valid both for \eqref{GP} and \eqref{NLS}. We define, following Martel \cite{Mar}, 
\begin{equation*}
V_{\rho}(t):=V_{\rho}(u(t))=2\int_{\R^3} \rho_R(x)|u(t,x)|^2\,dx,
\end{equation*}
where  
 $\rho,$ which is in particular a well-constructed  function depending only on the two variables  $\bar x=(x_1,x_2)$ which provides a localization in the exterior of a cylinder, parallel to the $x_3$ axis and with radius of size $|\bar x|\sim R.$   Here $|\bar x|$ clearly denotes $|\bar x|:=(x_1^2+x_2^2)^{1/2}.$
Moreover we consider the not-localized  function $x_3^2$ in order to obtain a virial-like estimate of the form
\begin{equation}\label{eq:intro-vir}
\frac{d^2}{dt^2}V_{\rho_R+x_3^2}(t)\leq \underbrace{4\int_{\R^3}|\nabla u(t)|^2 dx +6 \lambda_1\int_{\R^3}|u(t)|^4 dx }_{=:\,\mathcal G}+cR^{-2}+H_R(u(t)),
\end{equation}
where the error  $H_R$ is defined by
\[
\begin{aligned}
H_R(u(t))&=4\lambda_1 \int_{\R^3}F_R(\bar x)|u(t)|^4\, dx+2\lambda_2\int_{\R^3}\nabla\rho_R\cdot\nabla\left(K\ast|u(t)|^2\right)|u(t)|^2\,dx\\
&-4\lambda_2\int_{\R^3} x_3\partial_{x_3}\left(K\ast|u(t)|^2\right)|u(t)|^2\,dx
\end{aligned}
\]
and $F_R(\bar x)$ is a nonnegative function supported in the exterior of a cylinder of  radius of order $R$.   
\item In the case $\lambda_1<0$, $\lambda_2=0$ (namely NLS), the decay property of the $L^4$-norm of a function $f$ supported outside a cylinder of radius of order $R$,  more precisely the estimate  $\|f\|_{L^4(|\bar x|\geq R)}^4\lesssim R^{-1}\| f\|_{\dot H^1(\R^3)}^2$, together with the localized virial identities, implies finite time blow-up by a convexity argument. We underline that everything works well since we are able to prove that  $\| u(t)\|_{\dot H^1(\R^3)}^2$ controls either $G(u(t))$  or the remainder term $H_R(u(t)).$ Note that in this case the quantity  $\mathcal G$ in the r.h.s. of \eqref{eq:intro-vir} precisely defines $4G(u(t))$ in the context of \eqref{NLS}.
\item When $\lambda_2 \neq 0$ we have to deal with the effect of the dipolar interaction term -- incorporated in $H_R(u(t))$ in \eqref{eq:intro-vir} -- that is nonlocal and that is neither always positive nor always negative. As already pointed out, see \autoref{rem:k-rie}, $K\ast\cdot$ acts as the sum (up to some constant coefficients) of square of Riesz transforms.
Our strategy is to split $u$ (we omit the time dependence) by separating it in the interior and in the exterior of a cylinder, namely $u=u_i+u_o$ where 
\[u_i=\bold 1_{\{|\bar x|\leq CR\}}u \quad \hbox{ and } \quad u_o=\bold 1_{\{|\bar x|\geq CR\}}u,
\] 
and computing the interaction given by the dipolar term. Here $\bold 1$ denotes the indicator function on a measurable set. The problem here is that $K\ast|u_i|^2$ is not supported inside any cylinder.  A crucial tool is given by the pointwise estimate
\[|\bold 1_{\{|\bar x|\leq \gamma_1R\}}(x)\mathcal R_j^2[(1-\bold 1_{\{|\bar x|\leq \gamma_2R\}})f](x)|\leq CR^{-3}\bold 1_{\{|\bar x|\leq \gamma_1R\}}(x)\|f\|_{L^1(|\bar x|\geq\gamma_2R)}\]
where $\gamma_1$ and $\gamma_2$ are positive parameters satisfying $d:=\gamma_2-\gamma_1>0,$ see \autoref{lemma:in-out}.
\item In order to control the remainder term and  to make appear the   $6\lambda_2\int_{\R^3}(K\ast|u(t)|^2)|u(t)|^2\,dx$ term in \eqref{eq:intro-vir} that will yield to the whole quantity $4G(u(t)),$ see \eqref{GG}, we need to use the identity $2\int_{\R^3} x\cdot\nabla\left(K\ast f\right)f\,dx=-3\int_{\R^3}\left(K\ast f\right)f\,dx$. The latter follows from the relation $\xi\cdot \nabla_\xi\hat K=0.$  The difficulty here comes from the fact that the localization function used in the virial identities $\rho_R(\bar x)$ satisfies 
\[
\rho_R(\bar x)=
\begin{cases}
|\bar x|^2 &\hbox{ if }\, |\bar x|<R\\
\hbox{constant}  &\hbox{ if }\, |\bar x|>2R
\end{cases},
\] 
while the function $\rho=x_3^2$ is unbounded.  By observing that 
\[
\begin{aligned}
2\int_{\R^3} x_3\partial_{x_3}\left(K\ast|u_i|^2\right)|u_o|^2\,dx+2\int_{\R^3}x_3\partial_{x_3}\left(K\ast|u_o|^2\right)|u_i|^2\,dx\\
=-2\int_{\R^3}\left(K\ast |u_i|^2\right)|u_o|^2\,dx-2\int_{\R^3}\xi_3(\partial_{\xi_3}\hat K)\widehat{ |u_i|^2 }\bar{\widehat{|u_o|^2}}\,d\xi
\end{aligned}
\]
and that 
\[
\xi_3\partial_{\xi_3}\hat K=8\pi\frac{\xi_3^2(\xi_1^2+\xi_2^2)}{|\xi|^4}=8\pi\left(\frac{\xi_3^2}{|\xi|^2}-\frac{\xi_3^4}{|\xi|^4}\right)=8\pi \widehat{\mathcal{R}_3^2}-8\pi\widehat{\mathcal{R}_3^4},
\]
we reduce the problem to the estimate of $|\langle \mathcal R_3^4f,g\rangle_{L^2}|$ when $f$ is supported in $\{|\bar x|\geq \gamma_2R\}$ while $g$ is supported in $\{|\bar x|\leq \gamma_1R\},$ for some positive parameters $\gamma_1$ and $\gamma_2$ satisfying $d:=\gamma_2-\gamma_1>0.$ Here $\mathcal R_j^4$ denotes the fourth power of the Riesz transform, and $\widehat{\mathcal{R}_j^4}$ its symbol in Fourier space.
\item We compute $|\langle \mathcal R_3^4f,g\rangle_{L^2}|$ by means of some corollaries (see \autoref{lemma:decay-r4} and \autoref{lemma:decay-r2-2}) of a more general result from harmonic analysis, related to the representation of a Fourier operator $T$ whose symbol is  homogeneous of degree zero, in conjunction with the localization properties of the supports of functions where $T$ is acting on, see \autoref{thm:zero-degree}.

\item All the previous points will bring to the final estimate  
\[
\frac{d^2}{dt^2}V_{\rho_R+x_3^2}(t)\leq 4G(u(t))+\epsilon_{R}(u(t)),
\] 
where $G(u(t))\lesssim -\tilde\delta\|u(t)\|_{\dot H^1(\R^3)}^2$ and $\epsilon_R(u(t))\lesssim o_{R}(1)\|u(t)\|_{\dot H^1(\R^3)}^2,$ which in turn implies the finite time blow-up via a convexity argument, provided $R\gg1.$ 
\end{itemize}

\subsection{Notation and structure of the paper}
We collect here the notation used along the paper and we disclose how the paper is organized.
We work in the three dimensional space $\R^3,$ and for a vector $x=(x_1,x_2,x_3)\in\R^3$ we denote by $\bar x\in\R^2$ the vector $\bar x=(x_1,x_2)$ given by the first two components of $x\in\R^3.$ The differential operators $\nabla$ and $\nabla\cdot$ are the common gradient and divergence operator in $\R^3.$ When using the subscript $\bar x,$ i.e. $\nabla_{\bar x}$ or $\nabla_{\bar x}\cdot,$ we mean that we are considering them as operators on $\R^2$ with respect to the variables $(x_1,x_2)$ alone. The operator $\mathcal Ff(\xi)=\hat f(\xi)=\int e^{-ix\xi}f(x)\,dx$ is the standard Fourier Transform, $\mathcal F^{-1}$ being its inverse. $\mathcal R_j$ is the $j$-th Riesz transform defined vie the Fourier symbol $-i\frac{\xi_j}{|\xi|},$ i.e. $\mathcal R_jf(x)=\mathcal F^{-1}\left(-i\frac{\xi_j}{|\xi|}\hat f\right)(x).$ Powers of the Riesz transform are defined by means of powers of their symbols analogously.
For $1\leq p\leq \infty$ and $\Omega\subseteq\R^3,$  $L^p(\Omega)=L^p(\Omega;\mathbb C)$ are the classical Lebesgue spaces endowed with norm $\|f\|_{L^p}=\left(\int_{\Omega}|f(x)|^p\,dx\right)^{1/p}$ if $p\neq\infty$ or $\|f\|_{L^\infty}=\esssup_{x\in\Omega}|f(x)|$ for $p=\infty.$ When $\Omega=\R^3$ we simply write $L^p.$ For a function $f(x),$ $x\in\R^3,$ we denote $\|f\|_{L^p_{\bar x}}=\|f(\cdot, x_3)\|_{L^p_{\bar x}(\R^2)}$ and similarly for more general domains $\Omega\subset \R^2.$
We set $H^1=H^1(\mathbb R^3;\mathbb C):=\{f \, s.t. \,  \int_{\R^3}(1+|\xi|^2)|\hat f(\xi)|^2\,d\xi<\infty\}$ and its homogeneous version $\dot H^1=\dot H^1(\mathbb R^3;\mathbb C):=\{f \, s.t. \,  \int_{\R^3}|\xi|^2|\hat f(\xi)|^2\,d\xi<\infty\}$ endowed with their natural norms. Since we work on $\R^3,$ we simply denote $\int f\,dx=\int_{\R^3}f\,dx$ and often we write $\int f\,dx=\int_{\R}\int_{\R^2}f(\bar x,x_3)\,d\bar x\,dx_3.$ The expression $(f\ast g)(x):=\int f(x-y)g(y)\,dy$ denotes the convolution operator between $f$ and $g.$ The $L^2$ inner product between two function $f,g$ is denoted by $\langle f,g\rangle=\langle f,g\rangle_{L^2}:=\int fg\,dx.$  $\Re{z}$ and $\Im{z}$ are the common notations for the real and imaginary parts of a complex number $z.$ When the bar-symbol over-lines a complex-valued function, we mean the complex conjugate.  Given a measurable set $\mathcal A\subseteq\R^3,$ $\bold 1_{\mathcal A}(x)$ is the indicator function of $\mathcal A.$ Finally, given two quantities $A$ and $B,$ we denote $A \lesssim B$ ($A\gtrsim B,$ respectively) if there exists a positive constant $C$ such that $A \leq CB$ ($A\geq CB,$ respectively). If both  relations hold true, we write $A\sim B.$\\

In \autoref{sec:2} we prove the essential integral and pointwise estimates for powers of the Riesz transforms for suitably localized functions. In \autoref{sec:3} we discuss the geometry of the energy functional and we disclose several properties leading to the control of the functional $G$ in terms of $\|u\|_{\dot H^1}.$ In \autoref{sec:NLS} we prove the blow-up in finite time for the focusing cubic NLS stated in \autoref{thm:main-nls}, then we conclude with \autoref{sec:GPE} where we prove the main result of the paper, namely the finite time blow-up for the dipolar GPE. We collect in the \autoref{app:A} some useful identities used along the proofs in the paper. In  \autoref{app:B} we make a connection between the fourth power of the Riesz transform with the propagator associated to the linear parabolic biharmonic equation, which has its own interest, and it could be used to give an alternative proof for the integral decay estimates in \autoref{sec:2}.

\section{Localization properties of the dipolar kernel}\label{sec:2}

This section provides the first technical tools we need in order to prove our main result concerning the finite time blow-up for the GPE \eqref{GP}. In the next Lemmas we prove some decay estimates -- pointwise and integral estimates -- regarding the square and the 4-th power of the Riesz transforms for suitably localized functions. 
We prove this decay by employing a general harmonic analysis tool which gives an explicit characterization of homogeneous distributions.
Subsequently, we prove the  pointwise estimates for $\mathcal R_j^2$ by using the explicit representation of $\mathcal R^2_j$ in terms of the singular integral defined in the principal value sense.

\subsection{Integral estimates for $\mathcal R^4_j$}\label{subsect:R4}  We start with the integral estimates for the fourth power of the Riesz transform, and, as anticipated above,  we use a general result in harmonic analysis regarding the characterization of homogeneous distribution on $\R^n$ of degree $-n,$ coinciding with a regular function in $\R^n\setminus\{0\}.$ For our purposes, likely along the whole paper, we just consider $n=3.$ The main contribution of this section is as follows.
\begin{prop}\label{thm:zero-degree}
Let $T$ an operator defined by means of a Fourier symbol $m(\xi),$ which is smooth in $\R^3\setminus\{0\}$ and is a homogenous function of degree zero, i.e. $m(\lambda \xi)=m(\xi)$ for any $\lambda>0.$ For any couple of functions $f,g\in L^1$ having disjoint supports, we have the following estimate:
\begin{equation}\label{eq:int-decay-general}
|\langle Tf,g\rangle|\lesssim \left(\mbox{dist}(\mbox{supp}(f), \mbox{supp}(g))\right)^{-3}\|g\|_{L^1}\|f\|_{L^1}.
\end{equation}
\end{prop}
By observing that the symbols defining $\mathcal R_i^4$ and $\mathcal R_k^2\mathcal R_h^2$ are given by $m(\xi)=\frac{\xi_i^4}{|\xi|^4}$ and $m(\xi)=\frac{\xi_k^2 \xi_h^2}{|\xi|^4}$, respectively, which fulfil the hypothesis of \autoref{thm:zero-degree}, we straightforwardly have the following corollaries. It is worth mentioning that for our applications, we will consider as $f$ and $g$ some cut-off of the density of the mass $|u(t)|^2,$ which is clearly in $L^1$.
\begin{corollary}\label{lemma:decay-r4}
Assume that $f,g\in L^1$ and that $f$ is supported in $\{|\bar x|\geq \gamma_2R\}$ while $g$ is supported in $\{|\bar x|\leq \gamma_1R\},$ for some positive parameters $\gamma_1$ and $\gamma_2$ satisfying $d:=\gamma_2-\gamma_1>0.$ Then 
\begin{equation*}\label{eq:est-r4}
|\langle \mathcal R_i^4f,g\rangle|\lesssim R^{-3}\|g\|_{L^1}\|f\|_{L^1}.
\end{equation*}
\end{corollary}
\begin{corollary}\label{lemma:decay-r2-2}
Assume that $f,g\in L^1,$ and that $f$ is supported in $\{|\bar x|\geq \gamma_2R\}$ while $g$ is supported in $\{|\bar x|\leq \gamma_1R\},$ for some positive parameters $\gamma_1$ and $\gamma_2$ satisfying $d:=\gamma_2-\gamma_1>0.$ Then 
\begin{equation*}
|\langle \mathcal R^2_k\mathcal R_h^2 f,g\rangle|\lesssim R^{-3}\|g\|_{L^1}\|f\|_{L^1}.
\end{equation*}
\end{corollary}
\begin{proof}[Proof of \autoref{thm:zero-degree}] By definition, $\mathcal F(Tf)(\xi)=m(\xi)\hat f(\xi),$ and being $m$ a homogeneous symbol of degree zero, smooth away from the origin, we can invoke \cite[Proposition 2.4.7]{Grafakos} and we can claim the existence of a smooth function $\Omega$ on the sphere $\mathbb S^2,$ and a scalar $c\in\C$ such that 
\[
\mathcal F^{-1}m=\frac{1}{|x|^3}\Omega\left(\frac{x}{|x|}\right)+c\delta(x),
\]  
where $\delta$ is the Dirac delta at the origin. We recall that being $m$ a symbol of degree zero, the associated distribution  is homogeneous of degree $-3.$ 
Hence, 
\[
\begin{aligned}
\langle Tf,g\rangle&= \iint \frac{1}{|x-y|^3}\Omega\left(\frac{x-y}{|x-y|}\right)f(y)g(x)\,dy\,dx+c\iint \delta(x-y)f(y)g(x)\,dy\,dx\\
&=\iint \frac{1}{|x-y|^3}\Omega\left(\frac{x-y}{|x-y|}\right)f(y)g(x)\,dy\,dx,
\end{aligned}
\]
where in the last identity we used the disjointness of the supports of $f$ and $g$. Therefore,
\[
\begin{aligned}
|\langle Tf,g\rangle|&\leq \left(\mbox{dist}(\mbox{supp}(f), \mbox{supp}(g))\right)^{-3}\|\Omega\|_{L^\infty(\mathbb S^2)}\|f\|_{L^1}\|g\|_{L^1}\\
&\lesssim \left(\mbox{dist}(\mbox{supp}(f), \mbox{supp}(g))\right)^{-3}\|f\|_{L^1}\|g\|_{L^1} ,
\end{aligned}
\]
and the proof is concluded.
\end{proof}

\subsection{Pointwise estimate for $\mathcal R^2_j$.}
We turn now the attention to the square of the Riesz transform. In the subsequent results, we will use a cut-off function $\chi$ satisfying the following: $\chi(x)$ is a localization function supported in the cylinder $\{|\bar x|\leq 1\}$  which is nonnegative and bounded, with $\|\chi\|_{L^\infty}\leq 1.$ For a positive parameter $\gamma,$ we define by $\chi_{\{|\bar x|\leq \gamma R\}}$ the rescaled function $\chi(x/\gamma R)$ (hence $\chi_{\{|\bar x|\leq \gamma R\}}$ is bounded, positive and supported in the cylinder of radius $\gamma R$). The proof of the next lemmas is inspired by \cite{LW}.

\begin{lemma}\label{lemma:in-out} For any (regular) function $f$ the following pointwise estimate is satisfied: provided $d:=\gamma_2-\gamma_1>0,$ where $\gamma_1$ and $\gamma_2$ are positive parameters, there exists a universal constant 
$C>0$ such that 
\begin{equation}\label{eq:in-out}
|\chi_{\{|\bar x|\leq \gamma_1R\}}(x)\mathcal R_j^2[(1-\chi_{\{|\bar x|\leq \gamma_2R\}})f](x)|\leq CR^{-3}\chi_{\{|\bar x|\leq \gamma_1R\}}(x)\|f\|_{L^1(|\bar x|\geq\gamma_2R)}.
\end{equation}
\end{lemma}
\begin{remark} It is worth mentioning that  the use of the general \autoref{thm:zero-degree} yields an integral estimate for the operator $\mathcal R^2_j,$ which would suffices for our purposes later on in the paper, namely to close the convexity argument for the blow-up Theorem when the dipolar kernel in  \eqref{GP} is acting (i.e. $\lambda_2\neq0$). Nonetheless, we prefer to give the pointwise decay below as well, which is more refined than an integral estimate, and because it relies of the precise integral representation of the square of the Riesz transform, while \autoref{thm:zero-degree} holds true for any operator as in the hypothesis. 
\end{remark}
\begin{proof}[Proof of \autoref{lemma:in-out}] In the principal value sense, the square of the Riesz transform acts on a function $g$ as 
\[
\mathcal R_j^2g(x)=\iint \frac{x_j-y_j}{|x-y|^{3+1}}\frac{y_j-z_j}{|y-z|^{3+1}}g(z)\,dz\,dy.
\]
Let $g(x)=\chi_{\{|\bar x|\geq\gamma_2R\}}(x)f(x).$ Then
\[
\chi_{\{|\bar x|\leq \gamma_1R\}}(x)\mathcal R_j^2g(x)=\iint\left(\frac{y_j}{|y|^{4}}\frac{z_j-y_j}{|z-y|^{4}}\,dy\right)g(x-z)\,dz.
\] 
Since $g$ is supported in the exterior of a cylinder of radius $\gamma_2R,$ we can assume  $|\bar x-\bar z|\geq \gamma_2R,$ and for the function $\chi_{\{|\bar x|\leq \gamma_1R\}}$ is supported by definition in the cylinder of radius $\gamma_1R,$ we can assume  $|\bar x|\leq \gamma_1R:$ therefore we have that $|\bar z|\geq dR.$ This implies that $\{|\bar y|\leq \frac{d}{4}R\}\cap\{|\bar z-\bar y|\leq\frac12|\bar z|\}=\emptyset.$ Indeed, 
\begin{equation}\label{eq:y}
\frac12|\bar z|\geq|\bar z-\bar y|\geq |\bar z|-|\bar y| \implies |\bar y|\geq\frac12|\bar z|\geq \frac d2R,
\end{equation}
hence we have the following splitting:
\begin{equation}\label{eq:int:I}
\begin{aligned}
I=\int\frac{y_j}{|y|^{4}}\frac{z_1-y_1}{|z-y|^{4}}\,dy&=\int_{|\bar y|\leq\frac d4 R}\frac{y_j}{|y|^{4}}\frac{z_j-y_j}{|z-y|^{4}}\,dy\\
&+\int_{|\bar z-\bar y|\leq\frac12|\bar z|}\frac{y_j}{|y|^{4}}\frac{z_j-y_j}{|z-y|^{4}}\,dy\\
&+\int_{\{|\bar y|\geq \frac d4 R\}\cap\{|\bar z-\bar y|\geq \frac12|\bar z|\}}\frac{y_j}{|y|^{4}}\frac{z_j-y_j}{|z-y|^{4}}\,dy\\
&=\mathcal{I}+\mathcal{II}+\mathcal{III}.
\end{aligned}
\end{equation}
\noindent \textit{Estimate for the term $\mathcal{I}$.} Let us focus on the first integral $\mathcal{I}.$ The domain of integration of this integral is the cylinder of radius $R$ parallel to the $y_3$ axis.  Therefore
\[
\mathcal{I}=\mathcal{I}_1+\mathcal{I}_2=\int_{|y_3|\leq \frac d4 R}\int_{|\bar y|\leq \frac d4 R}\frac{y_j}{|y|^{4}}\frac{z_j-y_j}{|z-y|^{4}}\,d\bar y\,dy_3+\int_{|y_3|\geq \frac d4 R}\int_{|\bar y|\leq \frac d4 R}\frac{y_j}{|y|^{4}}\frac{z_j-y_j}{|z-y|^{4}}\,d\bar y\,dy_3.
\]
For the term $\mathcal{I}_1$ we first notice that 
\[
\int_{|y_3|\leq \frac d4 R}\int_{|\bar y|\leq\frac d4 R}\frac{y_j}{|y|^{4}}\,d\bar y\,dy_3=0
\]
since the domain is invariant under the change of variables $y\mapsto-y.$ Therefore 
\[
A_1=\int_{|y_3|\leq \frac d4 R}\int_{|\bar y|\leq \frac d4 R}\frac{y_j}{|y|^{4}}\left(\frac{z_j-y_j}{|z-y|^{4}}-\frac{z_j}{|z|}\right)\,d\bar y\,dy_3
\]
and we write 
\[
\frac{z_j-y_j}{|z-y|^{4}}-\frac{z_j}{|z|}=\int_0^1 h^\prime(s)\,ds
\]
where 
\[h(s)=\frac{z_j-sy_j}{|z-sy|^{4}}, \qquad s\in[0,1].
\] 
We have, by a straightforward calculation, that  
\[
h^\prime(s)=-\frac{y_j}{|z-sy|^{4}}+4\frac{z_j-sy_j}{|z-sy|^{4}}(z-sy)\cdot y
\]
and 
\[|h^\prime(s)|\lesssim \frac{|y|}{|z-sy|^{4}}.
\]
Hence, by observing that $|z-sy|\geq |\bar z-s\bar y|\geq |\bar z|-s|\bar y|\geq \frac{3d}{4} R$ as $s|\bar y|\leq \frac d4 R,$ we get that 
\[
\max_{s\in[0,1]}|h^\prime(s)|\lesssim R^{-4}|y|,  
\]
then 
\begin{equation}\label{est:A1}
\begin{aligned}
\mathcal{I}_1&=\int_{|y_3|\leq \frac d4 R}\int_{|\bar y|\leq \frac d4 R}\frac{y_j}{|y|^{4}}\left(\int_0^1 h^\prime(s)\,ds\right)\,d\bar y\,dy_3\lesssim R^{-4}\int_{|y_3|\leq \frac d4 R}\int_{|\bar y|\leq \frac d4 R}\frac{1}{|y|^{2}}\,d\bar y\,dy_3\\
&\lesssim R^{-4}\int_{|y|\leq\frac{ \sqrt2d}{4} R}\frac{1}{|y|^{2}}\,dy \lesssim R^{-3}.
\end{aligned}
\end{equation}

The term $\mathcal{I}_2$ can be estimated as follows:
\begin{equation*}
\begin{aligned}
\mathcal{I}_2&=\int_{|y_3|\geq \frac d4 R}\int_{|\bar y|\leq \frac d4 R}\frac{y_j}{|y|^{4}}\frac{z_j-y_j}{|z-y|^{4}}\,d\bar y\,dy_3\leq\int_{|y_3|\geq \frac d4 R}\int_{|\bar y|\leq \frac d4 R}\frac{1}{|y|^3}\frac{1}{|z-y|^3}\,d\bar y\,dy_3\\
&\leq \int_{|y_3|\geq \frac d4 R}\frac{1}{|y_3|^3}\int_{|\bar y|\leq \frac d4 R}\frac{1}{|\bar z-\bar y|^3}\,d\bar y\,dy_3\lesssim\int_{|y_3|\geq \frac d4 R}\frac{1}{|y_3|^3}\,dy_3\left( \frac{1}{R^3}\int_{|\bar y|\leq \frac d4 R}\,d\bar y\right)
\end{aligned}
\end{equation*}
where we used again the fact that if $|\bar y|\leq \frac d4 R$ then $|\bar z-\bar y|\geq \frac{3d}{4}R,$ hence we conclude with 
\begin{equation}\label{est:A2}
\mathcal{I}_2\lesssim R^{-2}R^{-1}=R^{-3}
\end{equation}
In conclusion, by summing up the two estimates \eqref{est:A1} and \eqref{est:A2}we get 
\begin{equation}\label{est:A}
\mathcal{I}\lesssim R^{-3}.
\end{equation}
\\
\noindent \textit{Estimate for the term $\mathcal{II}.$} We adopt a similar approach for the term $\mathcal{II}$ that we split in two further terms: 
\[
\begin{aligned}
\mathcal{II}&=\int_{|z_3-y_3|\leq |\bar z|}\int_{|\bar z-\bar y|\leq\frac12|\bar z|}\frac{y_j}{|y|^{4}}\frac{z_j-y_j}{|z-y|^{4}}\,dy+\int_{|z_3-y_3|\geq |\bar z|}\int_{|\bar z-\bar y|\leq\frac12|\bar z|}\frac{y_j}{|y|^{4}}\frac{z_j-y_j}{|z-y|^{4}}\,dy\\
&=\mathcal{II}_1+\mathcal{II}_2
\end{aligned}
\]  
We first notice that 
\[
\int_{|z_3-y_3|\leq |\bar z|}\int_{|\bar z-\bar y|\leq\frac12|\bar z|}\frac{z_j-y_j}{|z-y|^{4}}\,d\bar y\,dy_3=0
\]
since the domain is invariant under the change of variable $y\mapsto2z-y.$ \\
\noindent \textit{Estimate for the term $\mathcal{II}_1.$} Therefore 
\[
\begin{aligned}
\mathcal{II}_1&=\int_{|z_3-y_3|\leq |\bar z|}\int_{|\bar z-\bar y|\leq\frac12|\bar z|}\frac{z_j-y_j}{|z-y|^{4}}\left(
\frac{y_j}{|y|^{4}}-\frac{z_j}{|z|^{4}}\right)\,d\bar y\,dy_3\\
&=\int_{|z_3-y_3|\leq |\bar z|}\int_{|\bar z-\bar y|\leq\frac12|\bar z|}\frac{z_j-y_j}{|z-y|^{4}}\left(\int_0^1h^\prime(s)\,ds\right)\,d\bar y\,dy_3
\end{aligned}
\]
where 
\[
h(s)=\frac{sy_j+(1-s)z_j}{|sy_j+(1-s)z_j|^{4}},\qquad s\in[0,1].
\] 
We compute 
\[
h^\prime(s)=\frac{y_j-z_j}{|sy+(1-s)z|^{4}}+4\frac{sy_j+(1-s)z_j}{|sy+(1-s)z|^{5}}\frac{sy+(1-s)z}{|sy+(1-s)z|}\cdot(y-z)
\]
and hence 
\[
|h^\prime(s)|\lesssim \frac{|y-z|}{|sy+(1-s)z)|^{4}}.
\]
Now we observe, as $s|\bar y-\bar z|\leq\frac12|\bar z|,$ that $
|sy+(1-s)z|=|s(y-z)+z|\geq|s(\bar y-\bar z)+\bar z|\geq|\bar z|-s|\bar y-\bar z|\geq \frac12|\bar z| $
and then 
\[
\max_{s\in[0,1]}|h^\prime(s)|\lesssim \frac{|y-z|}{|\bar z|^{4}}
\]
which allows us to continue the estimate for $\mathcal{II}_1$ as follows: 
\begin{equation}\label{est:B1}
\begin{aligned}
\mathcal{II}_1&\lesssim \frac{1}{|\bar z|^{4}} \int_{|z_3-y_3|\leq |\bar z|}\int_{|\bar z-\bar y|\leq\frac12|\bar z|}\frac{1}{|z-y|^{2}}\,d\bar y\,dy_3\\
&\lesssim \frac{1}{|\bar z|^{4}} \int_{|z-y|\lesssim |\bar z|}\frac{1}{|z-y|^{2}}\,dy\lesssim |\bar z|^{-3}\lesssim R^{-3}.
\end{aligned}
\end{equation}

\noindent\textit{Estimate for the term $\mathcal{II}_2.$}  It remains to prove a suitable estimate for the remaining term $\mathcal{II}_2.$ 
We use \eqref{eq:y} and we estimate 
\begin{equation}\label{est:B2}
\begin{aligned}
\mathcal{II}_2&=\int_{|z_3-y_3|> |\bar z|}\int_{|\bar z-\bar y|\leq\frac12|\bar z|}\frac{y_j}{|y|^{4}}\frac{z_j-y_j}{|z-y|^{4}}\,d\bar y\,dy_3\\
&\leq\int_{|z_3-y_3|> |\bar z|}\int_{|\bar z-\bar y|\leq\frac12|\bar z|}\frac{1}{|y|^3}\frac{1}{|z-y|^3}\,d\bar y\,dy_3\\
&\leq\int_{|z_3-y_3|> |\bar z|}\int_{|\bar y|\geq\frac12|\bar z|}\frac{1}{|y|^3}\frac{1}{|z-y|^3}\,d\bar y\,dy_3\\
&\leq\int_{|z_3-y_3|> |\bar z|}\frac{1}{|z_3-y_3|^3}\,dy_3\int_{|\bar y|\geq\frac12|\bar z|}\frac{1}{|\bar y|^3}\,d\bar y\leq|\bar z|^{-2}|\bar z|^{-1}\lesssim R^{-3}.
\end{aligned}
\end{equation}
We conclude, by summing up \eqref{est:B1} and \eqref{est:B2}, that 
\begin{equation}\label{est:B}
\mathcal{II}\lesssim R^{-3}.
\end{equation}

\noindent \textit{Estimate for the term $\mathcal{III}.$} It is left to estimate the integral 
\[
\mathcal{III}=\int_{\{|\bar y|\geq \frac d4 R\}\cap\{|\bar z-\bar y|\geq \frac12|\bar z|\}}\frac{y_j}{|y|^{4}}\frac{z_j-y_j}{|z-y|^{4}}\,dy.
\]
By the Cauchy-Schwarz's inequality
\begin{equation}\label{est:C}
\begin{aligned}
\mathcal{III}&=\int_{\{|\bar y|\geq \frac d4 R\}\cap\{|\bar z-\bar y|\geq \frac12|\bar z|\}}\frac{y_j}{|y|^{4}}\frac{z_j-y_j}{|z-y|^{4}}\,dy\leq\int_{\{|\bar y|\geq \frac d4 R\}\cap\{|\bar z-\bar y|\geq \frac12|\bar z|\}}\frac{1}{|y|^3}\frac{1}{|z-y|^3}\,dy\\
&\leq\left(\int_{\{|\bar y|\geq \frac d4 R\}\cap\{|\bar z-\bar y|\geq \frac12|\bar z|\}}\frac{1}{|y|^{6}}\,dy\right)^{1/2}\left(\int_{\{|\bar y|\geq \frac d4 R\}\cap\{|\bar z-\bar y|\geq \frac12|\bar z|\}}\frac{1}{|z-y|^{6}}\,dy\right)^{1/2}\\
&\leq\left(\int_{\{|y|\geq \frac d4 R\}}\frac{1}{|y|^{6}}\,dy\right)^{1/2}\left(\int_{\{| z- y|\geq \frac12|\bar z|\}}\frac{1}{|z-y|^{6}}\,dy\right)^{1/2}\\
&\lesssim R^{-3/2}|\bar z|^{-3/2}\lesssim R^{-3}.
\end{aligned}
\end{equation}
The proof of the lemma is therefore concluded by observing that the integral $I$ defined in \eqref{eq:int:I} can be bounded, by using \eqref{est:A}, \eqref{est:B} and \eqref{est:C}, by
\[
I:=\mathcal{I}+\mathcal{II}+\mathcal{III}\lesssim R^{-3},
\]
and hence
\[
\begin{aligned}
|\chi_{\{|\bar x|\leq \gamma_1R\}}(x)\mathcal R_j^2g(x)|&=\chi_{\{|\bar x|\leq \gamma_1R\}}(x)\left|\iint\left(\frac{y_j}{|y|^{4}}\frac{z_j-y_j}{|z-y|^{4}}\,dy\right)g(x-z)\,dz\right|\\
&\lesssim R^{-3}\chi_{\{|\bar x|\leq \gamma_1R\}}(x)\int|g(x-z)|\,dz\\
&\lesssim R^{-3}\chi_{\{|\bar x|\leq \gamma_1R\}}(x)\|f\|_{L^1(|\bar x|\geq\gamma_2R)}
\end{aligned}
\] 
which is the estimate stated in \eqref{eq:in-out}.
\end{proof}

We have an estimate similar to \eqref{eq:in-out} if we localize inside a cylinder the function on which $R^2_j$ acts, and we then truncate everything with a function supported in the exterior of  another cylinder. 
\begin{lemma}\label{lemma:out-in} For any (regular) function $f$ the following pointwise estimate is satisfied: provided $d:=\gamma_1-\gamma_2>0,$ where $\gamma_1$ and $\gamma_2$ are positive parameters, there exists a  universal constant $C>0$ such that 
\begin{equation}\label{eq:out-in}
|(1-\chi_{|\bar x|\leq \gamma_1R})(x)\mathcal R_j^2[(\chi_{\{|\bar x|\leq \gamma_2R\}})f](x)|\leq CR^{-3}|(1-\chi_{\{|\bar x|\leq \gamma_1R\}})(x)|\|f\|_{L^1(|\bar x|\leq \gamma_2R)}.
\end{equation}
\end{lemma}
\begin{proof}
The proof is analogous to the one for \autoref{lemma:in-out}, so we skip the details. 
\end{proof}


\section{Variational structure of the energy functional and consequences}\label{sec:3}
We pass now to the discussion on the variational structure of the energy functional and its relation to the existence of standing waves for \eqref{GP}. The following arguments are valid in the same fashion for the NLS equation \eqref{NLS}, when the parameter $\lambda_2=0,$  even if it is worth mentioning that for NLS the existence of standing waves is nowadays classical. We recall the two different approaches to prove existence of Ground States for the GPE.  

The first strategy is due to Antonelli and Sparber, see \cite{AS}, where existence is proved  by means of minimization of the Weinstein functional
\begin{equation*}
J(v):=\frac{\|\nabla v\|^3_{L^2}\|v\|_{L^2}}{-\lambda_1\|v\|_{L^4}^4-\lambda_2\int (K\ast |v|^2)|v|^2\,dx}.
\end{equation*}
The alternative way, see the work of Jeanjean and  the first author \cite{BJ},  is based on topological methods, where the existence of Ground States is shown by means of the existence of critical points of the energy functional under the mass constraint \eqref{constraint1}. In the latter approach the parameter $\kappa$ which appears in \eqref{eq:maina} is found as Lagrange multiplier. 
Even if the energy functional is unbounded from below on $S(1)$, when restricting to states which are stationary for the evolution equation, i.e. they satisfy \eqref{eq:maina},  then the energy is bounded from below by a  positive constant. The latter constant, which corresponds to the mountain pass level, is reached. 
The mountain pass solutions therefore correspond  to the least energy states  (which are called Ground States, precisely). 
We pass now to the analysis of the geometry of the functional $E(u)$ on $S(c),$ and to this aim we introduce the $L^2$-preserving  scaling:
\begin{equation*}
u^\mu(x)=\mu^{3/2}u(\mu x), \quad \mu>0.
\end{equation*}
We report the next crucial lemma from \cite{BJ}. We recall the definition of $V(c)$ given in \eqref{def:vc}:
\[
V(c)=\{u\in H^1 \quad s.t. \quad \|u\|_{L^2}^2=c  \hbox{ and } G(u)=0\}.
\]
\begin{lemma}\label{lem:growth}{\cite[Lemma 3.3]{BJ}}
Suppose that  $u$ belongs to the manifold $S(c)$ and moreover that it satisfies $ \int (\lambda_1+\lambda_2 \hat K(\xi))(\widehat{ |u|^2})^2\,d\xi <0.$ Then the following properties hold true: 
\begin{itemize}
\item there exists a unique $\tilde\mu(u)>0$, such that $u^{\tilde\mu} \in V(c)$;
\item the map $\mu \mapsto E(u^{\mu})$ is concave on $[\tilde\mu, \infty)$;
\item $\tilde\mu(u)<1$ if and only if $G(u)<0$;
\item $\tilde\mu(u)=1$ if and only if $G(u)=0$;
\item the functional $G$ satisfies 
\begin{equation*}
G(u^\mu)
\begin{cases}
 >0,\quad \forall\, \mu \in (0,\tilde\mu(u))\\
 <0, \quad \forall\, \mu\in (\tilde\mu(u),+\infty)
\end{cases};
\end{equation*}
\item $E(u^{\mu})<E(u^{\tilde\mu})$, for any $\mu>0$ and $\mu \neq \tilde\mu$;
\item $\frac{d}{d \mu} E(u^{\mu})=\frac{1}{\mu}G(u^{\mu})$, $\forall \mu>0$.
\end{itemize}
\end{lemma}

With \autoref{lem:growth} at hand, we can prove the next proposition which basically shows the dichotomy between the scattering  and  blow-up for \eqref{GP} in terms of the quantities $\gamma(\|u_0\|_{L^2}^2)$ and $G(u_0).$
\begin{prop}\label{prop:coinc}
Suppose that the  initial datum $u_0$ satisfies $E(u_0)<\gamma(\|u_0\|_{L^2}^2)$ and $G(u_0)>0,$ then
\begin{equation}\label{HR3}
M(u_0)E(u_0)<M(Q)E(Q)
\end{equation}
and 
\begin{equation}\label{HR4}
\|u_0\|_{L^2}\|\nabla u_0\|_{L^2}<\| Q\|_{L^2}\|\nabla Q\|_{L^2}.
\end{equation}

\noindent Conversely, if the conditions expressed in \eqref{HR3} and \eqref{HR4} hold true, then the initial datum $u_0$ satisfies  $E(u_0)<\gamma(\|u_0\|_{L^2}^2)$ and $G(u_0)>0.$
\end{prop}
\begin{remark}\label{remark:JFA}
We point out that we gave the proof of the first  implication in our previous work \cite{BF19}, but we repeat it below as in that paper some steps were not rigorously justified (it is worth mentioning that the claim was however correct, and the validity of the result was not affected by that carelessness). 
\end{remark}
\begin{proof}
We start with the first implication. From the definition of the quantities in \eqref{energy}, \eqref{kinetic:en} and   \eqref{potential:en}, we straightforwardly obtain the identity 
\begin{equation}\label{eq:id-egt} 
E(u_0)-\frac 13 G(u_0)=\frac 16 T(u_0).
\end{equation}

\noindent Due to the scaling invariance properties of the Weinstein functional, we note that
$Q_{\mu}:=\mu Q(\mu x)$ is again a minimizer for the Weinstein functional with
\begin{equation*}
\begin{aligned}
\|Q_{\mu}\|_{L^2}^2&=\mu^{-1}\|Q\|_{L^2}^2, \\
\|\nabla Q_{\mu}\|_{L^2}^2&=\mu\|\nabla Q\|_{L^2}^2.
\end{aligned}
\end{equation*}
We notice that $Q(x)e^{i t}$ is a standing wave solution to the evolution equation and by the symmetry
of the equation it is well known that $Q_{\mu}e^{i \mu^2t}=\mu Q(\mu x)e^{i \mu^2t}$ is another   standing wave solution to
\begin{equation*}
-\frac12\Delta Q_{\mu}+\left(\lambda_1|Q_{\mu}|^2Q_{\mu}+\lambda_2(K\ast|Q_{\mu}|^2)Q_{\mu}\right)+\mu^2Q_{\mu}=0,
\end{equation*}
that necessarily satisfies
$G(Q_{\mu})=0$. Hence 
$E(Q_{\mu})=\frac 16  \|\nabla Q_{\mu}\|_{L^2}^2$. 

\noindent  Provided we choose the parameter $\mu$ such that $\|Q_{\mu}\|_{L^2}^2=\|u_0\|_{L^2}^2,$ i.e. $Q_\mu$ belongs to the constraint $S(c)$, $c=\|u_0\|_{L^2}^2,$ we get (using the hypothesis) 
\begin{equation}\label{eq:exp}
E(u_0)<\gamma(\|u_0\|_{L^2}^2)=\gamma(\|Q_\mu\|_{L^2}^2)=E(Q_{\mu}).
\end{equation}
From \eqref{eq:exp} we obtain  
\begin{equation*}
\|u_0\|_{L^2}^2 E(u_0)<\|Q\|_{L^2}^2E(Q),   
\end{equation*}
which corresponds to \eqref{HR3}. 
It is worth remarking how we can claim the equality in  \eqref{eq:exp} (this is precisely the clarification we do with respect to what we wrote in \autoref{remark:JFA}). It is crucial to notice that if $Q$ is a standing state (solving the elliptic equation
with a corresponding Lagrange multiplier), then $Q_{\mu}$ is a standing state  for any $\mu>0$. 
On the other hand if  $Q$  is not a standing state (i.e. it does not solve the elliptic equation
with any Lagrange multiplier), then $Q_{\mu}$ is not a standing state for any $\mu>0$.\\
If we  take two standing states with the same mass, let say $w$ and $v$, with their corresponding Lagrange multipliers, and such that
$E(w)<E(v)$, then $E(w_{\mu})<E(v_{\mu})$ for any $\mu>0$. This is evident by the fact that $E(w_\mu)=\frac 16 \|\nabla w_\mu\|_{L^2}^2=\frac\mu6 \|\nabla w\|_{L^2}^2=\frac\mu6E(w)$ for any $\mu>0$ (indeed $G(w_\mu)$ is always 0). Therefore if $E(w)<E(v)$ then $E(w_{\mu})<\frac\mu6\|\nabla v\|_{L^2}^2=E(v_{\mu})$.\\
This implies that in  the case of a Mountain Pass  solution, if one takes a standing wave $Q$ such that $E(Q)=\gamma (\|Q\|_{L^2}^2)$ then $E(Q_{\mu})=\gamma (\|Q_{\mu}\|_{L^2}^2)$.\\

We prove now the validity of the other condition. If $G(u_0)>0$ and $E(u_0)<\gamma(\|u_0\|_{L^2}^2)=E(Q_{\mu})$, then we have
\begin{equation*}
\frac 16  \|\nabla Q_{\mu}\|_{L^2}^2=E(Q_{\mu})>E(u_0)>E(u_0)-\frac 13 G(u_0)=\frac 16 \|\nabla u_0\|_{L^2}^2
\end{equation*}
and hence
\begin{equation*}
\|u_0\|_{L^2}\|\nabla u_0\|_{ L^2}< \|Q\|_{L^2}\|\nabla Q\|_{L^2}.
\end{equation*}

Let us prove the reverse implication. 

\noindent First of all we notice that if $Q$ is a minimizer for the Weinstein functional, then $E(Q)=\gamma(\|Q\|_{L^2}^2).$ We take a rescaling $Q_\mu$ of $Q$ such that $\|Q_\mu\|_{L^2}^2=\|u_0\|_{L^2}^2$ and as before we can claim that $E(Q_\mu)=\gamma(\|u_0\|_{L^2}^2).$ Therefore \eqref{HR3} implies 
\[
M(u_0)E(u_0)<M(Q)E(Q)=M(Q_\mu)E(Q_\mu)=M(Q_\mu)\gamma(\|u_0\|_{L^2}) \implies E(u_0)<\gamma(\|u_0\|_{L^2}).
\]
Let us focus on the statement ``$\eqref{HR4}\implies G(u_0)>0$''. Suppose  that $G(u_0)\leq0$ and consider $\tilde\mu$ such that $u^{\tilde\mu}\in V(\|u_0\|_{L^2}^2).$  By using \autoref{lem:growth} such  $\tilde\mu$ do exists, $G(u_0^{\tilde\mu})=0$ (by the very definition of $V(\|u_0\|_{L^2}^2)$) and $\tilde\mu\leq1.$ (In particular, if $G(u_0)=0$ then $\tilde\mu=1$.) Hence
\[
E(u^{\tilde\mu})=\frac16\|\nabla u^{\tilde\mu}\|_{L^2}^2=\frac{\tilde\mu}{6}\|\nabla u\|_{L^2}^2\leq\frac16\|\nabla u\|_{L^2}^2<\frac16\|\nabla Q\|_{L^2}^2=E(Q).
\]  
This concludes the proof since we got a function, $u^{\tilde\mu}$, such that $E(u^{\tilde\mu})<E(Q),$ which contradicts the minimality of $E(Q).$ 
\end{proof}
\begin{remark}
It is straightforward to see that in \autoref{prop:coinc}, provided $E(u_0)<\gamma(\|u_0\|_{L^2}^2),$  the condition \eqref{HR4} replaced by 
\begin{equation}\label{HR5}
\|u_0\|_{L^2}\|\nabla u_0\|_{L^2}>\| Q\|_{L^2}\|\nabla Q\|_{L^2}
\end{equation}
will imply that $G(u_0)<0,$ and conversely  \eqref{HR5} is satisfied provided  we assume $G(u_0)<0.$ Hence $E(u_0)<\gamma(\|u_0\|_{L^2}^2)$ and $G(u_0)>0$ or $G(u_0)<0$ give the dichotomy between scattering and blow-up for \eqref{GP}. It is worth mentioning that when restricting to the cubic NLS case, the latter ones are the same described by Holmer and Roudenko in \cite{HR}, namely \eqref{HR3}, \eqref{HR4}, and \eqref{HR5}. 
\end{remark}
\begin{lemma}\label{lem:2.2}
If the initial datum  $u_0$ satisfies  $E(u_0)<\gamma(\|u_0\|_{L^2}^2)$ and $G(u_0)<0$ then $G(u(t))<0$ for any $t\in(-T_{min}, T_{max}).$ More precisely, there exists a positive constant $\delta>0$ such that $G(u(t))\leq-\delta$ for any $t\in(-T_{min}, T_{max}).$ 
\end{lemma}
\begin{proof}
Suppose that $G(u(t))>0$ for some time $t\in(-T_{min}, T_{max});$ then by the continuity in time of the function $G(u(t))$ there exists $\tilde t$ such that $G(u(\tilde t))=0.$ By definition we have therefore $\gamma(\|u_0\|_{L^2}^2)\leq E(u(\tilde t))=E(u_0)$ which is a contradiction with respect to the definition of $\gamma (\|u_0\|_{L^2}^2).$ 

We  now prove the uniform bound from below away from zero. We simply denote $u=u(t).$ By \autoref{lem:growth} -- third claim -- there exists $\tilde\mu\in(0,1)$ such that $G(u^{\tilde\mu})=0.$ Then 
\[
E(u_0)-E(u^{\tilde\mu})=(1-\tilde\mu)\frac{d}{d\mu}E(u^\mu)\vert_{\mu=\bar\mu}
\] 
for some $\bar\mu\in(\tilde\mu,1),$ and due to the concavity of $\mu\mapsto E(u^\mu)$ -- see \autoref{lem:growth}, second claim -- we have that
\[
E(u_0)-E(u^{\tilde\mu})=(1-\tilde\mu)\frac{d}{d\mu}E(u^\mu)\vert_{\mu=\bar\mu}\geq (1-\tilde\mu) \frac{d}{d\mu}E(u^\mu)\vert_{\mu=1}=(1-\tilde\mu)G(u)
\] 
where in the last equality we used the last claim of \autoref{lem:growth}. Hence 
\[
G(u(t))\leq (1-\tilde\mu)^{-1}\left( E(u_0)-E(u^{\tilde\mu})\right)\leq (1-\tilde\mu)^{-1} (E(u_0)-\gamma(c)).
\]
The proof is complete with $\delta=(1-\tilde\mu)^{-1}( \gamma(c)-E(u_0)).$
\end{proof}
The previous Lemma implies the pointwise-in-time bound for the function $G(u(t)),$ by means of the homogeneous Sobolev $\dot H^1$-norm of $u(t).$
\begin{lemma}\label{lem:2.3}
There exists $\alpha>0$ such that $G(u(t))\leq-\alpha \|u(t)\|_{\dot H^1}^2$ for any $t\in(-T_{min},T_{max}).$
\end{lemma}
\begin{proof}
It follows from \autoref{lem:2.2} and the identity \eqref{eq:id-egt}, that $\|u(t)\|_{\dot H^1}^2>6E(u_0).$  By exploiting again the identity \eqref{eq:id-egt} we write $\|u(t)\|_{\dot H^1}^2=6E-2G(u(t)),$ so we have 
\begin{equation}\label{eq:G-al}
G(u(t))+\alpha\|u(t)\|_{\dot H^1}^2=(1-2\alpha)G(u(t))+6\alpha E.
\end{equation}
As we have that $G\leq-\delta$, it is straightforward  to see that for $\alpha\ll1$ the claim follows, since for $\alpha$ sufficiently small, the r.h.s. of \eqref{eq:G-al} is bounded by $-\delta/2$. 
\end{proof}
We give the following simple consequence of the previous Lemma.
\begin{corollary}\label{coro1}
There exists a positive constant $c>0$ such that 
\[
\inf_{t\in(-T_{min},T_{max})}\|u(t)\|_{\dot H^1}\geq c.
\]
\end{corollary}
\begin{proof}
Suppose that there exists a sequence of times $\{t_n\}_{n\in\N}$ such that $\lim_{n\to\infty}\|u(t_n)\|_{\dot H^1}=0.$ Then by the Gagliardo-Nirenberg's inequality $\lim_{n\to\infty}\|u(t_n)\|_{L^4}=0$ as well. But then $G(u(t_n))\to0,$ since by the $L^2\mapsto L^2$ property of the dipolar kernel 
\[
|G(u(t_n))|\lesssim \|u(t_n)\|_{\dot H^1}^2+\|u(t_n)\|_{L^4}^4\to0,
\]
which contradicts \autoref{lem:2.2}
\end{proof}


\section{Blow-up for the focusing cubic NLS}\label{sec:NLS}
In this chapter we prove \autoref{thm:main-nls} regarding the NLS \eqref{NLS}. We first construct a suitable cut-off function localizing in the exterior of a cylinder parallel to the $x_3$ axis, which we then plug it in the virial identities below. This cut-off function will be used also in the proof of \autoref{thm:main} regarding the blow-up in finite time for solutions to the GPE, but the estimates we need in order to control some remainders in the presence of the dipolar kernel are much more involved. Therefore we prefer to state the result for the focusing cubic NLS which has its own interest, passing to the analysis of the equation \eqref{GP} in the next chapter. Previous papers in literature about the formation of singularities in finite time for $L^2$-supercritical focusing NLS in 3D are \cite{KRRT, HR07, GM, DWZ, RS}, besides the already cited \cite{Mar, Gla, OT}.

We now introduce some classical virial identities, valid both for \eqref{GP} and \eqref{NLS}. Let $u(t)$ be a solution to \eqref{GP}, which also corresponds to a solution to \eqref{NLS} for $\lambda_2=0$ and $\lambda_1<0;$ then we define 
\begin{equation}\label{vir0}
V_{\rho}(t)=2\int \rho(x)|u(t,x)|^2\,dx,
\end{equation} 
$\rho(x)$ being a sufficiently smooth function which justifies the following formal and standard computations:   
\begin{equation}\label{virial1}
\begin{aligned}
\frac{d}{dt}V_{\rho}(t)&=2\Im\left\{\int\nabla \rho\cdot\nabla u\bar u\,dx\right\},
\end{aligned}
\end{equation}
where we used the equation satisfied by $u=u(t).$
By using \eqref{virial1} and again the equation solved by $u(t),$ we have 
\begin{equation*}
\begin{aligned}
\frac{d^2}{dt^2}V_{\rho}(t)&=2 \int\left(\nabla^2\rho\cdot\nabla u\right)\cdot\nabla\bar u\,dx-\frac{1}{2}\int\Delta^2\rho|u|^2\,dx\\
&+\lambda_1\int\Delta \rho|u|^4\,dx-2\lambda_2\int\nabla \rho\cdot\nabla\left(K\ast|u|^2\right)|u|^2\,dx.
\end{aligned}
\end{equation*}
Therefore for the rescaled function $\rho_R(x):=R^2\rho(x/R)$ we get 
\begin{equation}\label{virial:g}
\begin{aligned}
\frac{d^2}{dt^2}V_{\rho_R}(t)&=2\int\left(\nabla^2 \rho\left(\frac xR\right)\cdot\nabla u\right)\cdot\nabla\bar u\,dx-\frac{1}{2R^2}\int\Delta^2\rho\left(\frac xR\right)|u|^2\,dx\\
&+\lambda_1\int\Delta\rho\left(\frac xR\right)|u|^4\,dx-2\lambda_2R\int\nabla\rho\left(\frac xR\right)\cdot\nabla\left(K\ast|u|^2\right)|u|^2\,dx.
\end{aligned}
\end{equation}

Let us precisely build the (rescaled) function $\rho,$ which is in particular a radial cut-off function depending only on the two variables  $\bar x=(x_1,x_2):$ 
\[
\rho_R(x)=\rho_R(|\bar x|)=R^2\rho(|\bar x|/R)=R^2\psi(|\bar x|^2/R^2)
\]   
where
\[
\psi(r)=r-\int_0^r(r-s)\eta(s)\,ds,
\]
and the function $\eta:\R\mapsto\R^+_0$ is a function satisfying the following properties: it is a nonnegative, regular function with unitary mean, namely:
\begin{equation*}
\begin{cases}
\eta\in\mathcal C^\infty_c(\R;\R^+_0)\\
\eta(s)=0 \quad  \hbox{ for } s\leq1  \hbox{ and } s\geq2\\
\int_\R \eta(s)\,ds=1
\end{cases}.
\end{equation*}  
We have, for $i,j\in\{1,2\}$ and  $\delta_{ij}$ being the usual Kronecker symbol,  
\[
\partial_{x_i}\rho_R(x)=2x_i\psi^\prime(|\bar x|^2/R^2)
\]
and
\[
\partial_{x_j,x_i}^2\rho_R(x)=2\delta_{ij}\psi^\prime(|\bar x|^2/R^2)+\frac{4x_ix_j}{R^2}\psi^{\prime\prime}(|\bar x|^2/R^2),
\]
hence
\[
\begin{aligned}
\left(\nabla^2\rho\left(\frac xR\right)\nabla u\right)\cdot\nabla\bar u&=2\psi^\prime(|\bar x|^2/R^2)|\nabla_{\bar x}u|^2+\frac{4}{R^2}|\bar x|^2\psi^{\prime\prime}(|\bar x|^2/R^2)|\nabla_{\bar x}u|^2\\
&=2|\nabla_{\bar x}u|^2\left( \psi^\prime(|\bar x|^2/R^2)+\frac{2}{R^2}|\bar x|^2\psi^{\prime\prime}(|\bar x|^2/R^2)\right)
\end{aligned}
\]
and 
\[
\Delta \rho_R=4\psi^\prime(|\bar x|^2/R^2)+\frac{4}{R^2}|\bar x|^2\psi^{\prime\prime}(|\bar x|^2/R^2).
\]
We observe moreover that $\Delta^2\rho\in L^\infty$ and therefore we can estimate $R^{-2}\int\Delta^2\rho(x/R)|u|^2\,dx $ as
\begin{equation}\label{small:mass}
R^{-2}\int\Delta^2\rho(x/R)|u|^2\,dx \leq C(\|\Delta\rho\|_{L^\infty}) R^{-2}\|u(t)\|_{L^2}^2=CMR^{-2},
\end{equation}
by using the conservation of the mass. 

We compute explicitly the following:
\begin{equation}\label{eq:pos}
-F_R(\bar x):=\psi^\prime(|\bar x|^2/R^2)+\frac{|\bar x|^2}{R^2}\psi^{\prime\prime}(|\bar x|^2/R^2)-1=-\int_0^{|\bar x|^2/R^2}\eta(s)\,ds-\frac{|\bar x|^2}{R^2}\eta(|\bar x|^2/R^2)\leq0.
\end{equation}
We observe that by its construction 
\begin{equation}\label{eq:notzero}
L^\infty\ni F_R
\begin{cases}
\,=0 &\quad\textit{for any } \quad |\bar x|\leq R \\
\,>0 &\quad\textit{for any } \quad |\bar x|> R
\end{cases}.
\end{equation}
and the boundedness is uniform with respect to $R\geq1.$ 

Moreover we consider the not-localized virial identity in the $x_3$ direction; we namely plug in \eqref{vir0} the function $\rho=Ax_3^2,$ $A$ being a positive constant,  and we simply get, 
\begin{equation}\label{vir:x2}
\frac{d^2}{dt^2}V_{Ax_3^2}(t)=4A\int|\partial_{x_3}u|^2\,dx+2A\lambda_1\int|u|^4\,dx
-4A\lambda_2\int x_3\partial_{x_3}\left(K\ast|u|^2\right)|u|^2\,dx.
\end{equation}
For $A=1$ and by plugging $\lambda_2=0$ in \eqref{virial:g}, we are going to prove that 
\begin{equation}\label{eq:intro-vir-2}
\begin{aligned}
\frac{d^2}{dt^2}V_{\rho_R+x_3^2}(t)&\leq 4\int|\nabla u|^2 dx +6 \lambda_1\int |u|^4 dx+\tilde H_R(u(t))\\
&=4G(u(t))+\tilde H_R(u(t)),
\end{aligned}
\end{equation}
where 
\begin{equation}\label{eq:o1}
\tilde H_R(u(t))=o_R(1)+o_R(1)\|u(t)\|_{\dot H^1}^2;
\end{equation}
therefore \autoref{lem:2.3}, \autoref{coro1} coupled with a convexity argument will yield to the result stated in \autoref{thm:main-nls}, as long as $R\gg1.$ Note that for $\lambda_2=0,$ $G(u)$ is reduced to $4\int|\nabla u|^2 dx +6 \lambda_1\int |u|^4 dx.$

\subsection{Estimate of the remainder $H_R(u(t))$}
Once fixed  $\lambda_2=0,$  and for $\lambda_1<0,$ by suitably manipulating \eqref{virial:g} by adding and subtracting some quantity, it can be rewritten as 
\begin{equation*}
\begin{aligned}
\frac{d^2}{dt^2}V_{\rho_R}(t)&=4\int\left(\psi^\prime(|\bar x|^2/R^2)+\frac{2}{R^2}|\bar x|^2\psi^{\prime\prime}(|\bar x|^2/R^2)\right)|\nabla_{\bar x}u|^2\,dx\\
&+\lambda_1\int\Delta\rho( x/R)|u|^4\,dx-\frac{1}{2R^2}\int\Delta^2\rho( x/R)|u|^2\,dx\\
&=4\int\left(\psi^\prime(|\bar x|^2/R^2)+\frac{2}{R^2}|\bar x|^2\psi^{\prime\prime}(|\bar x|^2/R^2)+1-1\right)|\nabla_{\bar x}u|^2\,dx\\
&+4\int|\partial_{x_3}u|^2\,dx-4\int|\partial_{x_3}u|^2\,dx-\frac{1}{2R^2}\int\Delta^2\rho( x/R)|u|^2\,dx\\
&+4\lambda_1\int\left( \psi^\prime(|\bar x|^2/R^2)+\frac{|\bar x|^2}{R^2}\psi^{\prime\prime}(|\bar x|^2/R^2)+\frac32-\frac32\right)|u|^4\,dx,
\end{aligned}
\end{equation*}
and by using the definition of $G(u),$ see \eqref{GG}, taking into account that the potential energy $P(u)$ is simply given by $\lambda_1\int|u|^4\,dx,$ see \eqref{potential:en} for $\lambda_2=0,$ we get
\begin{equation*}
\begin{aligned}
\frac{d^2}{dt^2}V_{\rho_R}(t)&=4G(u(t))-\frac{1}{2R^2}\int\Delta^2\rho( x/R)|u|^2\,dx-4\int|\partial_{x_3}u|^2\,dx\\
&+4\int\left(\psi^\prime(|\bar x|^2/R^2)+\frac{2}{R^2}|\bar x|^2\psi^{\prime\prime}(|\bar x|^2/R^2)-1\right)|\nabla_{\bar x}u|^2\,dx\\
&+4\lambda_1\int\left( \psi^\prime(|\bar x|^2/R^2)+\frac{|\bar x|^2}{R^2}\psi^{\prime\prime}(|\bar x|^2/R^2)-\frac32\right)|u|^4\,dx
\end{aligned}
\end{equation*}
and 
\[
\frac{d^2}{dt^2}V_{Ax_3^2}(t)=4A\int|\partial_{x_3}u|^2\,dx
+2A\lambda_1\int|u|^4\,dx,\] 
therefore by summing up the two terms we get
\begin{equation}\label{virial:nls2}
\begin{aligned}
\frac{d^2}{dt^2} V_{\rho_R+Ax_3^2}(t)&=4G(u(t))-\frac{1}{2R^2}\int\Delta^2\rho( x/R)|u|^2\,dx-4(1-A)\int|\partial_{x_3}u|^2\,dx\\
&+4\int\left(\psi^\prime(|\bar x|^2/R^2)+\frac{2}{R^2}|\bar x|^2\psi^{\prime\prime}(|\bar x|^2/R^2)-1\right)|\nabla_{\bar x}u|^2\,dx\\
&+4\lambda_1\int\left( \psi^\prime(|\bar x|^2/R^2)+\frac{|\bar x|^2}{R^2}\psi^{\prime\prime}(|\bar x|^2/R^2)-\frac32+\frac A2\right)|u|^4\,dx\\
&\leq 4G(u(t))+CR^{-2}\\
&+4|\lambda_1|\int\left(1-\psi^\prime(|\bar x|^2/R^2)-\frac{|\bar x|^2}{R^2}\psi^{\prime\prime}(|\bar x|^2/R^2)\right)|u|^4\,dx,
\end{aligned}
\end{equation}
where we have set $A=1,$ we have used the definition of $G(u),$ \eqref{small:mass}, and \eqref{eq:pos}. Note that, even if it is a simple computation, we preferred to state the virial identity \eqref{vir:x2} with a general constant $A$ to emphasize how the choice $A=1$ is precisely done to make appear -- in the last line of \eqref{virial:nls2} -- the function $F_R$ defined above.
At this point \eqref{virial:nls2} reduces to 
\begin{equation*}
\begin{aligned}
\frac{d^2}{dt^2}V_{\rho_R+x_3^2}(t)&\leq 4G(u(t))+CR^{-2}+4|\lambda_1| \iint_{|\bar x|>R} F_R(\bar x)|u|^4\,d\bar x\,dx_3.
\end{aligned}
\end{equation*}
We estimate, in the spirit of Martel \cite{Mar},
\begin{equation}\label{0term}
\iint F_R(\bar x)|u|^4\,d\bar x\,dx_3\leq \int \|F_R|u|^2\|_{L^\infty_{\bar x}}\|u\|_{L^2_{\bar x}}^2\,dx_3\leq \|u\|_{L^\infty_{x_3}L^2_{\bar x}}^2\int\|F_R|u|^2\|_{L^\infty_{\bar x}}\,dx_3,
\end{equation}
where the norm in the $\bar x$ variable are meant in the domain $\{|\bar x|\geq R\}$ due to \eqref{eq:notzero} .
We  now use the Strauss embedding for a radial function $g(\bar x), \,\bar x\in\R^2$ and $g\in H^1,$ see \cite{CO},
\begin{equation}\label{strauss-emb}
\|g\|_{L^\infty_{\bar x}(|\bar x|>R)}\lesssim R^{-1/2}\|g\|_{L^2_{\bar x}}^{\frac 12}\|g\|_{\dot H^1_{\bar x}}^{\frac 12},
\end{equation}
hence by recalling that $F_R$ is bounded in space uniformly in $R\geq1,$ we obtain, by using \eqref{strauss-emb} and the Cauchy-Schwarz's inequality, that
\begin{equation}\label{1term}
\begin{aligned}
\int\|F_R|u|^2\|_{L^\infty_{\bar x}}\,dx_3&\lesssim \int \||u|^2\|_{L^\infty_{\bar x}(|\bar x|\geq R)}\,dx_3\lesssim \int \|u\|_{L^\infty_{\bar x}(|\bar x|\geq R)}^2\,dx_3\lesssim R^{-1}\|u\|_{\dot H^1}.
\end{aligned}
\end{equation}
On the other hand, by calling $g(x_3)=\int_{\R^2}|u|^2(\bar x,x_3)\,d\bar x$ we have
\[
g(x_3)=\int_{-\infty}^{x_3} \partial_sg(s)\,ds=2\Re \int_{-\infty}^{x_3} \left( \int_{\R^2} \bar u \partial_s u\,d\bar x \right)ds\leq 2 \int |u||\nabla u|\, dx \leq2 \sqrt M\| u\|_{\dot H^1}
\]
and then 
\begin{equation}\label{2term}
\|u\|_{L^\infty_{x_3} L^2_y}^2\lesssim \| u\|_{\dot H^1}.
\end{equation}
By glueing up \eqref{1term} and \eqref{2term} we get that  \eqref{0term} satisfies
\begin{equation}\label{est:F4}
\iint F_R|u|^4\,d\bar x\,dx_3\lesssim R^{-1}\|u\|_{\dot H^1}^2.
\end{equation}
We conclude that for some $\alpha>0$
\begin{equation*}
\frac{d^2}{dt^2}V_{\rho_R+x_3^2}(t)\lesssim G(u(t))+R^{-1}+R^{-1}\| u(t)\|_{\dot H^1}^2 \lesssim-\frac{\alpha}{2}\|u(t)\|_{\dot H^1}^2\lesssim -1,
\end{equation*}
where in  last step we used  \autoref{lem:2.2}, \autoref{lem:2.3}, \autoref{coro1} and we have chosen $R\gg1.$ 
We can eventually conclude  that $T_{max}<+\infty$ by a convexity argument. Indeed, it implies that there exists a time $T_0$ such that 
\[\int(\rho_R+x_3^2)|u(t)|^2\,dx\to 0 \quad \hbox{ as }\quad t\to T_0.
\]
We observe the following:  by using  the Weyl-Heisenberg's inequality $\|xf\|_{L^2}\|f\|_{\dot H^1}\gtrsim\|f\|_{L^2}^2$ in 1D we get  
\[
\|u_0\|_{L^2}^2=\|u(t)\|_{L^2}^2=\int_{\R^2}\int_{\R}|u(t)|^2\,dx_3\,d\bar x\lesssim \int_{\R^2} \left(\int_{\R}|\partial_{x_3}u(t)|^2\,dx_3\right)^{1/2} \left(\int_{\R}x_3^2|u(t)|^2\,dx_3\right)^{1/2}\,d\bar x
\]
and by using the Cauchy-Schwarz's inequality we conclude with the estimate 
\[
\|x_3u(t)\|_{L^2}\|\partial_{x_3}u(t)\|_{L^2}\gtrsim\|u_0\|_{L^2}^2>0.
\]
Then $\|u(t)\|_{\dot H^1}\to\infty$ as $t\to T_0.$
The proof of \autoref{thm:main-nls} is therefore concluded.

\section{Blow-up for the dipolar GPE}\label{sec:GPE}

This last chapter is devoted to the proof of \autoref{thm:main}. With respect to the NLS equation, when dealing with the equation \eqref{GP} governing a dipolar BEC,  we get an additional term to be estimated in the sum of the virial identies \eqref{virial:g}  and \eqref{vir:x2}, namely the sum of the two terms involving $\lambda_2.$ 
It is worth mentioning that all the terms that we have shown to be small  in the NLS case will remain the same. What we are going to prove is that for the dipolar GPE \eqref{GP} we have a virial-type estimate of the form (cf. with \eqref{eq:intro-vir-2}, \eqref{eq:o1}) 
\begin{equation}\label{eq:intro-vir-3}
\frac{d^2}{dt^2}V_{\rho_R+x_3^2}(t)\leq 4\int|\nabla u|^2 dx +6 \lambda_1\int |u|^4 dx+H_R(u(t))
\end{equation}
where 
\begin{equation}\label{eq:HVW}
H_R=\tilde H_R-2\lambda_2\left(\int\nabla\rho_R\cdot\nabla\left(K\ast|u|^2\right)|u|^2\,dx+2\int x_3\partial_{x_3}\left(K\ast|u|^2\right)|u|^2\,dx\right).
\end{equation}
As already proved in \autoref{sec:NLS}, see \eqref{est:F4},
\begin{equation}\label{eq:recall-H}
\tilde H_R(u(t))=o_R(1)+o_R(1)\|u(t)\|_{\dot H^1}^2,
\end{equation}  
so we are going to show that 
\begin{equation}\label{eq:V+W}
\mathcal V+\mathcal W:=\int\nabla\rho_R\cdot\nabla\left(K\ast|u|^2\right)|u|^2\,dx+2\int x_3\partial_{x_3}\left(K\ast|u|^2\right)|u|^2\,dx
\end{equation}
will contribute for a term $-3\int(K\ast|u|^2)|u|^2\,dx$ that will yield to the quantity $4G(u)$ (see the definition \eqref{GG}) when  summed to the first two terms in the r.h.s. of \eqref{eq:intro-vir-3}, plus a remainder term which must be proved to be small, again of the type $o_R(1)+o_R(1)\|u\|_{\dot H^1}^2.$

\subsection{Estimate of the remainder $H_R(u(t)).$ The $\mathcal V$ term}
By its definition, we get that the function $\rho_R$ fulfils  
\[
\partial_{x_i}\rho_R(x)=2x_i\psi^\prime(|\bar x|^2/R^2)=2\bar x\left(1-\int_0^{|\bar x|^2/R^2}\eta(s)\,ds\right)=\begin{cases}
2\bar x \quad &if \quad |\bar x|^2/R^2\leq 1\\
0 \quad &if \quad  |\bar x|^2/R^2>2
\end{cases},
\]
hence $\mbox{supp}\nabla_{\bar x}\rho_R$ is contained in the cylinder of radius $\sqrt 2 R.$

We split the function $u$ by partitioning the whole space in the region inside and the region outside  a cylinder, namely we write  $u=u_i+u_o$ where 
\[u_i=\bold 1_{\{|\bar x|\leq 4R\}}u \quad \hbox{ and } \quad u_o=\bold 1_{\{|\bar x|\geq 4R\}}u.
\] 
Since $\mbox{supp} \nabla \rho_R\cap \mbox{supp}\, u_o=\emptyset$ we get
\[
\begin{aligned}
\mathcal V&:=\int\nabla_{\bar x}\rho_R\cdot\nabla_{\bar x}\left(K\ast|u|^2\right)|u|^2\,dx\\
&=\int\nabla_{\bar x}\rho_R\cdot\nabla_{\bar x}\left(K\ast|u_o|^2\right)|u_i|^2\,dx-\int\nabla_{\bar x}\rho_R\cdot\nabla_{\bar x}\left(K\ast|u_i|^2\right)|u_i|^2\,dx\\
&=\mathcal V_{o,i}+\mathcal V_{i,i}.
\end{aligned}
\]
\noindent\emph{Estimate for the term $\mathcal V_{o,i}$.} The term $\mathcal V_{o,i}$ can be estimated in this way: by integrating by parts,
\[
\begin{aligned}
\mathcal V_{o,i}&=\int\nabla_{\bar x}\rho_R\cdot\nabla_{\bar x}\left(K\ast|u_o|^2\right)|u_i|^2\,dx\\
&=-\int\Delta_{\bar x}\rho_R\left(K\ast|u_o|^2\right)|u_i|^2\,dx-\int\nabla_{\bar x}\rho_R\cdot\nabla_{\bar x}\left(|u_i|^2\right)\left(K\ast|u_o|^2\right)\,dx=\mathcal V_{o,i}^\prime+\mathcal V_{o,i}^{\prime\prime}
\end{aligned}
\]
and, by using \autoref{lemma:in-out}, in particular the pointwise estimate \eqref{eq:in-out}, we obtain
\[
|\mathcal V_{o,i}^\prime|\lesssim\int \bold 1_{\{|\bar x|\leq \sqrt 2R\}}\left|K\ast|u_o|^2\right||u_i|^2\,dx\lesssim R^{-3}\|u_o\|_{L^2}^2\|u_i\|_{L^2}^2\lesssim R^{-3}\|u\|_{L^2}^4\lesssim R^{-3}.
\]
Similarly, by using that  $|\nabla_{\bar x}\rho_R|\lesssim R$ on its support in conjunction with the Cauchy-Schwarz's inequality, we get
\[
\begin{aligned}
|\mathcal V_{o,i}^{\prime\prime}|&\lesssim\int \bold 1_{\{|\bar x|\leq \sqrt 2R\}}\left|\nabla_{\bar x}\rho_R\right|\left|K\ast|u_o|^2\right| |u_i||\nabla_{\bar x} u_i|\,dx\\
&\lesssim R \int \bold 1_{\{|\bar x|\leq \sqrt 2R\}}\left|K\ast|u_o|^2\right| |u_i||\nabla_{\bar x} u_i|\,dx\lesssim R^{-2}\|u_o\|_{L^2}^2\|u_i\|_{L^2}\|u_i\|_{\dot H^1}\lesssim R^{-2}\|u\|_{\dot H^1}
\end{aligned}
\]
and then the estimate for $\mathcal V_{o,i}$ is concluded by summing up the two bounds above:
\begin{equation}\label{est:Voi}
|\mathcal V_{o,i}|\lesssim R^{-2}\|u\|_{\dot H^1}+R^{-3}.
\end{equation}
\noindent \emph{Estimate for the term $\mathcal V_{i,i}$.} We analyse the term $\mathcal V_{i,i}.$ We do a further splitting and we introduce another localization function.
By setting up $\tilde\rho_R=\rho_R-|\bar x|^2$ we can write
\begin{equation}\label{eq:split-V}
\begin{aligned}
\mathcal V_{i,i}&=\int\nabla_{\bar x}\rho_R\cdot\nabla_{\bar x}\left(K\ast|u_i|^2\right)|u_i|^2\,dx\\
&=\int\nabla_{\bar x}\tilde\rho_R\cdot\nabla_{\bar x}\left(K\ast|u_i|^2\right)|u_i|^2\,dx+2\int\bar x\cdot\nabla_{\bar x}\left(K\ast|u_i|^2\right)|u_i|^2\,dx=\mathcal V_{i,i}^\prime+\mathcal V_{i,i}^{\prime\prime}.
\end{aligned}
\end{equation}
We further localize the function $u_i$  by splitting again $u_i$ as $u_i=w_{i,i}+w_{i,o},$ where 
\[
w_{i,i}=\bold 1_{\{|\bar x|\leq R/10\}}u_i \quad \hbox{ and } \quad w_{i,o}=\bold 1_{\{|\bar x|\geq R/10\}}u_i=\bold 1_{\{R/10\leq|\bar x|\leq 4R\}}u.
\]
By using  the fact that $\mbox{supp}\nabla_{\bar x}\tilde\rho_R$ is contained in $\{|\bar x|\geq R\},$ then $\mbox{supp}\nabla_{\bar x}\tilde\rho_R\cap \{|\bar x|\leq R/10\}=\emptyset,$ we write
\[
\begin{aligned}
\mathcal V_{i,i}^\prime&=\int\nabla_{\bar x}\tilde\rho_R\cdot\nabla_{\bar x}\left(K\ast|w_{i,i}|^2\right)|w_{i,o}|^2\,dx\\
&+\int\nabla_{\bar x}\tilde\rho_R\cdot\nabla_{\bar x}\left(K\ast|w_{i,o}|^2\right)|w_{i,o}|^2\,dx=\mathcal A+\mathcal B.
\end{aligned}
\]
Now, similarly to the term $\mathcal V_{o,i},$ by integrating by parts and by using in this case \autoref{lemma:out-in}, and precisely the pointwise estimate \eqref{eq:out-in}, we have 
\begin{equation}\label{eq:sec5:A-1}
\begin{aligned}
\mathcal A&=\int\nabla_{\bar x}\tilde\rho_R\cdot\nabla_{\bar x}\left(K\ast|w_{i,i}|^2\right)|w_{i,o}|^2\,dx=-\int_{\{R\leq|\bar x|\leq 4R\}}\Delta_{\bar x}\tilde\rho_R\left(K\ast|w_{i,i}|^2\right)|w_{i,o}|^2\,dx\\
&-\int_{\{R\leq|\bar x|\leq 4R\}}\nabla_{\bar x}\tilde\rho_R\cdot\nabla_{\bar x}\left(|w_{i,o}|^2\right)\left(K\ast|w_{i,i}|^2\right)\,dx\\
&\lesssim R^{-3}\|w_{i,i}\|_{L^2}^2\|w_{i,o}\|_{L^2}^2+R^{-2}\|w_{i,i}\|_{L^2}^2\|w_{i,o}\|_{L^2}\|w_{i,o}\|_{\dot H^1}\lesssim R^{-3}+R^{-2}\|u\|_{\dot H^1}.
\end{aligned}
\end{equation}
Note that we used the following facts: $|\Delta_{\bar x}\tilde\rho_R|\lesssim 1,$ and by recalling that  $u_i=\bold 1_{\{|\bar x|\leq 4R\}}u,$ we infer that $|\nabla\tilde\rho_R|\lesssim R$ on $\{|\bar x|\leq 4R\}.$ \\
\noindent \emph{Estimate for the term $\mathcal B$.} Hence it remains to prove a suitable estimate for the term
\begin{equation}\label{eq:sec5:B-2}
\mathcal B=\int_{\{R\leq|\bar x|\leq \sqrt2R\}}\nabla_{\bar x}\tilde\rho_R\cdot\nabla_{\bar x}\left(K\ast|w_{i,o}|^2\right)|w_{i,o}|^2\,dx.
\end{equation}
By setting $g=|w_{i,o}|^2$ and making use of the Plancherel identity we get 
\begin{equation*}
\int\nabla_{\bar x}\tilde\rho_R\cdot\nabla_{\bar x}\left(K\ast|w_{i,o}|^2\right)|w_{i,o}|^2\,dx=\int \widehat{g\nabla_{\bar x}\tilde\rho}(\xi)\cdot\bar{\xi}\hat K\bar{\hat g}\,d\xi.
\end{equation*}
By recalling the precise expression for the Fourier transform of the dipolar kernel, see \eqref{kernel:fou}, we note that $\hat K$ is a linear combination of the symbols defining the square of the $j$-th Riesz transform $\mathcal R^2_j:$
\[
\hat K(\xi)=\frac{4\pi}{3}\frac{2\xi_3^2-\xi_2^2-\xi_1^2}{|\xi|^2}=\sum_{j=1}^3\alpha_j\frac{\xi^2_j}{|\xi|^2}.
\]
Consider therefore the generic term $\int \widehat {\nabla_{\bar x}\tilde\rho_R g}(\xi)\cdot\bar \xi \frac{\xi_j^2}{|\xi|^2}\bar{\hat g}\,d\xi$; by adding and subtracting $\frac{\eta_j\bar\eta}{|\eta|}$, we have  
\begin{equation}\label{term:generic}
\begin{aligned}
\int \widehat {g\nabla_{\bar x}\tilde\rho_R}(\xi)\cdot\bar\xi \frac{\xi_j^2}{|\xi|^2}\bar{\hat g}(\xi)\,d\xi&=\int (\widehat{\nabla_{\bar x}\tilde\rho_R}\ast\hat g)(\xi)\cdot\bar\xi \frac{\xi_j^2}{|\xi|^2}\bar{\hat g}(\xi)\,d\xi\\
&=\iint \hat g(\eta)\widehat{\nabla_{\bar x}\tilde\rho_R}(\xi-\eta)\cdot\left(\frac{\xi_j\bar\xi}{|\xi|}+\frac{\eta_j\bar\eta}{|\eta|}-\frac{\eta_j\bar\eta}{|\eta|}\right)\frac{\xi_j}{|\xi|}\hat g(\xi)\,d\eta\,d\xi\\
&=\int \nabla_{\bar x}\tilde\rho_R\cdot\nabla_{\bar x}(\mathcal R_j g)(x)\mathcal R_j\bar g(x)\,dx\\
&+\iint \hat g(\eta)\widehat{\nabla_{\bar x}\tilde\rho_R}(\xi-\eta)\cdot\left(\frac{\xi_j\bar\xi}{|\xi|}-\frac{\eta_j\bar\eta}{|\eta|}\right)\frac{\xi_j}{|\xi|}\hat g(\xi)\,d\eta\,d\xi\\
&=-\frac12\int\Delta_{\bar x}\tilde\rho_R|\mathcal R_j g(x)|^2\,dx\\
&+\iint \hat g(\eta)\widehat{\nabla_{\bar x}\tilde\rho_R}(\xi-\eta)\cdot\left(\frac{\xi_j\bar\xi}{|\xi|}-\frac{\eta_j\bar\eta}{|\eta|}\right)\frac{\xi_j}{|\xi|}\hat g(\xi)\,d\eta\,d\xi.
\end{aligned}
\end{equation}
The first term in the r.h.s. of \eqref{term:generic} is simply estimated by 
\begin{equation}\label{eq:sec5:B-1}
\|u\|^4_{L^4(|\bar x|\geq R/10)}\lesssim R^{-1}\|u\|_{\dot H^1}^2
\end{equation}
due to the $L^2\mapsto L^2$ continuity property of the Riesz transform and \eqref{est:F4}. For the second term in the r.h.s. of \eqref{term:generic} we proceed  in this way. First of all suppose that $j=1,2;$ then 
\begin{equation*}
\begin{aligned}
&\iint \hat g(\eta)\widehat{\nabla_{\bar x}\tilde\rho_R}(\xi-\eta)\cdot\left(\frac{\xi_j\bar\xi}{|\xi|}-\frac{\eta_j\bar\eta}{|\eta|}\right)\frac{\xi_j}{|\xi|}\hat g(\xi)\,d\eta\,d\xi\\
=&\int \frac{\xi_j}{|\xi|}\hat g(\xi)\iint\hat g(\bar \eta, \eta_3)\delta(\xi_3-\eta_3)\widehat{\tilde\rho_R}(\bar\xi-\bar\eta)(\bar\xi-\bar\eta)\cdot\left(\frac{\xi_j\bar\xi}{|\xi|}-\frac{\eta_j\bar\eta}{|\eta|}\right)\,d\bar\eta\,d\eta_3\,d\xi\\
=&\int \frac{\xi_j}{|\xi|}\hat g(\xi)\int\hat g(\bar \eta, \xi_3)\widehat{\tilde\rho_R}(\bar\xi-\bar\eta)(\bar\xi-\bar\eta)\cdot\underbrace{\left(\frac{\xi_j\bar\xi}{|\xi|}-\frac{\eta_j\bar\eta}{\sqrt{|\bar\eta|^2+\xi_3^2}}\right)}_{=:\,\vec F_j(\bar\xi)-\vec F_j(\bar\eta)}\,d\bar\eta\,d\xi. 
\end{aligned}
\end{equation*}
If instead $j=3$ we have 
\begin{equation*}
\begin{aligned}
&\iint \hat g(\eta)\widehat{\nabla_{\bar x}\tilde\rho_R}(\xi-\eta)\cdot\left(\frac{\xi_3\bar\xi}{|\xi|}-\frac{\eta_3\bar\eta}{|\eta|}\right)\frac{\xi_3}{|\xi|}\hat g(\xi)\,d\eta\,d\xi\\
=&\int \frac{\xi_3}{|\xi|}\hat g(\xi)\iint\hat g(\bar \eta, \eta_3)\delta(\xi_3-\eta_3)\widehat{\tilde\rho_R}(\bar\xi-\bar\eta)(\bar\xi-\bar\eta)\cdot\left(\frac{\xi_3\bar\xi}{|\xi|}-\frac{\eta_3\bar\eta}{|\eta|}\right)\,d\bar\eta\,d\eta_3\,d\xi\\
=&\int \frac{\xi_3}{|\xi|}\hat g(\xi)\int\hat g(\bar \eta, \xi_3)\widehat{\tilde\rho_R}(\bar\xi-\bar\eta)(\bar\xi-\bar\eta)\cdot\underbrace{\left(\frac{\xi_3\bar\xi}{|\xi|}-\frac{\xi_3\bar\eta}{\sqrt{|\bar\eta|^2+\xi_3^2}}\right)}_{=:\,\vec F_3(\bar\xi)-\vec F_3(\bar\eta)}\,d\bar\eta\,d\xi. 
\end{aligned}
\end{equation*}
where 
\[
\vec F_j(\bar v)=\frac{v_j \bar v}{\sqrt{|\bar v|^2+\xi^2_3}}=\left(\frac{v_jv_1}{\sqrt{|\bar v|^2+\xi^2_3}},\frac{v_j v_2}{\sqrt{|\bar v|^2+\xi^2_3}} \right),\quad \bar v=(v_1,v_2)\in\R^2,\quad j=1,2
\]
and 
\[
\vec F_3(\bar v)=\frac{\xi_3 \bar v}{\sqrt{|\bar v|^2+\xi^2_3}}=\left(\frac{\xi_3v_1}{\sqrt{|\bar v|^2+\xi^2_3}},\frac{\xi_3 v_2}{\sqrt{|\bar v|^2+\xi^2_3}} \right),\quad \bar v=(v_1,v_2)\in\R^2.
\]
We notice that the Jacobian $J_{\vec F_j}(\bar v)$ of $\vec F_j$ are uniformly bounded, namely $\left|J_{\vec F_j}(\bar v)\right|\lesssim 1$ for any $\bar v\in \R^2,$ for $j=1,2,3$; therefore we can bound the last term in the r.h.s. of \eqref{term:generic} by 
\begin{equation}\label{eq:fin-four}
\begin{aligned}
\int_{\R^3}|\hat g(\xi)|&\int_{\R^2}|\hat g(\bar \eta, \xi_3)|\left|\widehat{\tilde\rho_R}(\bar\xi-\bar\eta)\right||\bar\xi-\bar\eta|^2\,d\bar\eta\,d\xi\\
&\leq\int_{\R^3}|\hat g(\xi)|\int_{\R^2}|\hat g(\bar \eta, \xi_3)|\left|\widehat{\Delta_{\bar x}\rho_R}(\bar\xi-\bar\eta)\right|\,d\bar\eta\,d\xi\\
&+\int_{\R^3}|\hat g(\xi)|\int_{\R^2}|\hat g(\bar \eta, \xi_3)|\left|\widehat{\Delta_{\bar x}|\bar x|^2}(\bar\xi-\bar\eta)\right|\,d\bar\eta\,d\xi\\
&=\int_{\R^3}|\hat g(\xi)|\left( |\hat g(\cdot, \xi_3)|\ast h_R\right)(\bar \xi)\,d\xi+\int_{\R^3}|\hat g(\xi)|\left(|\hat g(\cdot, \xi_3)|\ast 4\delta\right)(\bar \xi)\,d\xi\\
&=\int_{\R^3}|\hat g(\xi)|\left( |\hat g(\cdot, \xi_3)|\ast h_R\right)(\bar \xi)\,d\xi+4\int_{\R^3}|\hat g(\xi)|^2\,d\xi
\end{aligned}
\end{equation}
where we defined $h_R=\left|\widehat{\Delta_{\bar x}\rho_R}\right|.$  We continue in this way: in the first  term in the r.h.s. of \eqref{eq:fin-four} we first apply the Cauchy-Schwarz's inequality and then the Young's inequality for convolutions with respect to $d\bar\xi,$ and eventually  the Cauchy-Schwarz's inequality with respect to $d\xi_3$ to obtain  
\begin{equation*}
\begin{aligned}
\int_{\R^3}|\hat g(\xi)|\left( |\hat g(\cdot, \xi_3)|\ast h_R\right)(\bar \xi)\,d\xi&=\int_{\R}\int_{\R^2}|\hat g(\bar\xi, \xi_3)|\left( |\hat g(\cdot, \xi_3)|\ast h_R\right)(\bar \xi)\,d\bar\xi\,d\xi_3\\
&\leq\int_{\R}\|\hat g(\cdot,\xi_3)\|_{L^2
_{\bar\xi}}\left\| |\hat g(\cdot, \xi_3)|\ast h_R \right\|_{L^2_{\bar\xi}}\,d\xi_3\\
&\leq \|h_R\|_{L^1_{\bar\xi}}\int_{\R}\|\hat g(\cdot, \xi_3)\|_{L^2_{\bar\xi}}^2\,d\xi_3= \|h_R\|_{L^1_{\bar\xi}}\|\hat g\|_{L^2}^2= \|h_R\|_{L^1_{\bar\xi}}\| g\|_{L^2}^2,
\end{aligned}
\end{equation*}
where in the last step we used the Plancherel identity. 
Since by definition the $L^1$-norm of $h_R$ is the $L^1$-norm of the Fourier transform of $\Delta_{\bar x}\rho_R,$ and the latter is a smooth and compactly supported function by its very construction, then its Fourier transform is still in $L^1_{\bar\xi}$ uniformly in $R.$  
Hence the estimate for the last integral in the r.h.s. of \eqref{term:generic}, recalling that $g=|w_{i,o}|^2=|\bold 1_{\{|\bar x|\geq R/10\}}|^2,$ can be concluded with 
\begin{equation}\label{eq:B-3}
\iint \hat g(\eta)\widehat{\nabla_{\bar x}\tilde\rho_R}(\xi-\eta)\cdot\left(\frac{\xi_j\bar\xi}{|\xi|}-\frac{\eta_j\bar\eta}{|\eta|}\right)\frac{\xi_j}{|\xi|}\hat g(\xi)\,d\eta\,d\xi\lesssim \|u\|_{L^4(|\bar x|\geq R/10)}^4\lesssim R^{-1}\|u\|_{\dot H^1}^2,
\end{equation}
where we used \eqref{est:F4} for the last inequality.
By summing up \eqref{eq:sec5:B-1} and \eqref{eq:B-3} we give the bound for the whole $\mathcal B$ term in \eqref{eq:sec5:B-2}:
\begin{equation}\label{eq:B-4}
\mathcal B\lesssim R^{-1}\|u\|_{\dot H^1}^2.
\end{equation} 
The estimate  \eqref{eq:B-4} in conjunction with \eqref{eq:sec5:A-1} conclude the estimate for $\mathcal V_{i,i}^\prime:$
\begin{equation}\label{eq:B-5}
\mathcal V_{i,i}^\prime \lesssim R^{-1}\|u\|_{\dot H^1}^2.
\end{equation} 

To end up with the full estimate leading to the proof of the main theorem, we are left to study the remaining term $\mathcal V_{i,i}^{\prime\prime}$ together with the term $\mathcal W$ as defined  in \eqref{eq:V+W}. 

\subsection{Estimate of the remainder $H_R(u(t))$. The $\mathcal V_{i,i}^{\prime\prime}+\mathcal W$ term} In order to conclude the estimate for the term $\mathcal V+\mathcal W,$ we are left to control the last term $\mathcal V_{i,i}^{\prime\prime}$ coming from the splitting in \eqref{eq:split-V}, and the term $\mathcal W,$ see \eqref{eq:V+W}, that we did not  handle so far.  A straightforward computation gives 
\begin{equation}\label{eq:Vpp-W}
\begin{aligned}
\mathcal V_{i,i}^{\prime\prime}+\mathcal W&=2\int\bar x\cdot\nabla_{\bar x}\left(K\ast|u_i|^2\right)|u_i|^2\,dx+2\int x_3\partial_{x_3}\left(K\ast |u|^2\right)|u|^2\,dx\\
&=2\int x\cdot\nabla\left(K\ast|u_i|^2\right)|u_i|^2\,dx+2\int x_3\partial_{x_3}\left(K\ast|u_i|^2\right)|u_o|^2\,dx\\
&+2\int x_3\partial_{x_3}\left(K\ast|u_o|^2\right)|u_i|^2\,dx+2\int x_3\partial_{x_3}\left(K\ast|u_o|^2\right)|u_o|^2\,dx\\
&=-3\int \left(K\ast|u_i|^2\right)|u_i|^2\,dx+2\int x_3\partial_{x_3}\left(K\ast|u_i|^2\right)|u_o|^2\,dx\\
&+2\int x_3\partial_{x_3}\left(K\ast|u_o|^2\right)|u_i|^2\,dx+2\int x_3\partial_{x_3}\left(K\ast|u_o|^2\right)|u_o|^2\,dx
\end{aligned}
\end{equation}
where we used the identity \eqref{eq:app-1}. By means of \eqref{eq:app-2} with $f=g=|u_o|^2$ we obtain 
\begin{equation}\label{eq:term:oo}
2\int x_3\partial_{x_3}\left(K\ast|u_o|^2\right)|u_o|^2\,dx=-\int\left(K\ast |u_o|^2\right)|u_o|^2\,dx-\int\xi_3(\partial_{\xi_3}\hat K)\widehat{ |u_o|^2}\overline{ \widehat {|u_o|^2}}\,d\xi
\end{equation}
while, again by means of \eqref{eq:app-2}, we write 
\begin{equation}\label{eq:last?}
\begin{aligned}
2\int x_3\partial_{x_3}\left(K\ast|u_i|^2\right)|u_o|^2\,dx+2\int x_3\partial_{x_3}\left(K\ast|u_o|^2\right)|u_i|^2\,dx\\
=-2\int\left(K\ast |u_i|^2\right)|u_o|^2\,dx-2\int\xi_3(\partial_{\xi_3}\hat K)\widehat{ |u_i|^2 }\overline{\widehat{|u_o|^2}}\,d\xi.
\end{aligned}
\end{equation}
We explicitly write  $\xi_3\partial_{\xi_3}\hat K$ and we observe that is bounded:
\begin{equation}\label{eq:der-ker-bound1}
\xi_3\partial_{\xi_3}\hat K=8\pi\frac{\xi_3^2(\xi_1^2+\xi_2^2)}{|\xi|^4}\in L^\infty_\xi;
\end{equation}
hence \eqref{eq:term:oo} is simply estimated by 
\[
2\int x_3\partial_{x_3}\left(K\ast|u_o|^2\right)|u_o|^2\,dx\lesssim \|u_o\|_{L^4}^4\lesssim R^{-1}\|u\|_{\dot H^1}^2
\]
by using the $L^2\mapsto L^2$ continuity of the dipolar kernel for the first integral,  while for the second integral we used the boundedness property \eqref{eq:der-ker-bound1} together with the Plancherel identity, and \eqref{est:F4}.

It remains to handle \eqref{eq:last?}. First of all we note that 
\begin{equation*}
\xi_3\partial_{\xi_3}\hat K=8\pi\frac{\xi_3^2(\xi_1^2+\xi_2^2)}{|\xi|^4}=8\pi\left(\frac{\xi_3^2}{|\xi|^2}-\frac{\xi_3^4}{|\xi|^4}\right)=8\pi \widehat{\mathcal{R}_3^2}-8\pi\widehat{\mathcal{R}_3^4},
\end{equation*}
where, with abuse of notation, we write $\widehat{\mathcal{R}_3^2}$ and $\widehat{\mathcal{R}_3^4}$ to denote the symbols of $\mathcal{R}_3^2$ and $\mathcal{R}_3^4,$ respectively. This in turn implies that the r.h.s. of \eqref{eq:last?} can be rewritten, for some constants $\alpha_1,\alpha_2, \alpha_3,$ as 
\[
-16\pi\int \left(\mathcal R_3^4(|u_i|^2)\right)|u_o|^2\,dx+\sum_{j=1}^3\alpha_j\int\left(\mathcal R_j^2( |u_i|^2)\right)|u_o|^2\,dx.
\]
By splitting $u_i=w_{i,i}+w_{i,o}$ with $w_{i,i}=\bold 1_{\{|\bar x|\leq R/10\}}u_i=\bold 1_{\{|\bar x|\leq R/10\}}u$ and  $w_{i,o}=\bold 1_{\{|\bar x|\geq R/10\}}u_i=\bold 1_{\{R/10\leq |\bar x|\leq 4R\}}u$ we decompose each of the terms in the sum involving  $\mathcal R_j^2$'s as 
\[
\begin{aligned}
\int\left(\mathcal R_j^2( |u_i|^2)\right)|u_o|^2\,dx&=\int\left(\mathcal R_j^2(|w_{i,i}|^2)\right)|u_o|^2\,dx+\int\left(\mathcal R_j^2( |w_{i,o}|^2)\right)|u_o|^2\,dx
\end{aligned}
\]
and by using the Cauchy-Schwarz's inequality, the $L^2\mapsto L^2$ continuity of the Riesz transform, the localization properties of $w_{i,o}$ and $u_{o},$  and \eqref{est:F4} we obtain
\begin{equation}\label{eq:wio-uo}
\int\left(\mathcal R_j^2( |w_{i,o}|^2)\right)|u_o|^2\,dx\lesssim \|w_{i,o}\|_{L^4}^2\|u_o\|_{L^4}^2\lesssim \|u\|_{L^4(|\bar x|\geq R/10)}^4\lesssim R^{-1}\|u\|_{\dot H^1}^2.
\end{equation}
By using \eqref{eq:out-in} we estimate
\begin{equation*}
\int\left(\mathcal R_j^2( |w_{i,i}|^2)\right)|u_o|^2\,dx\lesssim R^{-3}\|w_{i,i}\|_{L^2}^2\|u_o\|^2_{L^2}\lesssim R^{-3}.
\end{equation*}
With the same decomposition of the function $u_i$ we separate 
\[
\int \left(\mathcal R_3^4(|u_i|^2)\right)|u_o|^2\,dx=\int \left(\mathcal R_3^4(|w_{i,i}|^2)\right)|u_o|^2\,dx+\int \left(\mathcal R_3^4(|w_{i,o}|^2)\right)|u_o|^2\,dx;
\] 
then, similarly to \eqref{eq:wio-uo} we  have
\[
\int \left(\mathcal R_3^4(|w_{i,o}|^2)\right)|u_o|^2\,dx\lesssim R^{-1}\|u\|_{\dot H^1}^2,
\]
while by using \autoref{lemma:decay-r4} we can bound 
\[
\int \left(\mathcal R_3^4(|w_{i,i}|^2)\right)|u_o|^2\,dx\lesssim R^{-3}\|w_{i,i}\|_{L^2}^2\|u_o\|_{L^2}^2\lesssim R^{-3}.
\]

The remaining term in \eqref{eq:Vpp-W} is $-3\int \left(K\ast|u_i|^2\right)|u_i|^2\,dx.$ By adding and subtracting $u$ to $u_i$ we can infer, by similar computations done before, that
\begin{equation}\label{eq:very-last}
\begin{aligned}
-3\int \left(K\ast|u_i|^2\right)|u_i|^2\,dx&=-3\int \left(K\ast|u|^2\right)|u|^2\,dx+\epsilon(R, u)\\
\epsilon(R,u)&\lesssim o_R(1)+R^{-1}\|u\|_{\dot H^1}^2.
\end{aligned}
\end{equation}

\subsection{Proof of \autoref{thm:main}}

We can now summarize all the previous contributions towards the final step of the proof. We have, starting from \eqref{eq:intro-vir-3}-\eqref{eq:HVW}, that 
\begin{equation*}
\begin{aligned}
\frac{d^2}{dt^2}V_{\rho_R+x_3^2}(t)&\leq 4\int|\nabla u|^2 dx +6 \lambda_1\int |u|^4 dx+\tilde H_R(u(t))-2\lambda_2(\mathcal V+\mathcal W)\\
&=4\int|\nabla u|^2 dx +6 \lambda_1\int |u|^4 dx+\tilde H_R(u(t))-2\lambda_2(\mathcal V_{o,i}+\mathcal V_{i,i}^\prime+\mathcal V_{i,i}^{\prime\prime}+\mathcal W).
\end{aligned}
\end{equation*}
Thanks to the estimates \eqref{est:Voi}, \eqref{eq:recall-H},  \eqref{eq:B-5}, and \eqref{eq:very-last} which provide a control for the terms $\mathcal V_{o,i},$ $\tilde H_R(u(t)),$ $\mathcal V_{i,i}^\prime,$ and $\mathcal V_{i,i}^{\prime\prime}+\mathcal W,$ respectively, we end up with 
\begin{equation*}
\frac{d^2}{dt^2}V_{\rho_R+x_3^2}(t)\leq 4G(u(t))+o_R(1)+o_R(1)\|u(t)\|_{\dot H^1}^2.
\end{equation*}
Therefore by means of \autoref{lem:2.2}, \autoref{lem:2.3}, and \autoref{coro1}, provided $R\gg1$ we infer that $T_{min} =T_{max}=+\infty$ via a convexity argument. The proof of \autoref{thm:main} is now complete.


\appendix
\section{Identities for  convolution with the dipolar Kernel}\label{app:A}
We recall here some useful identities often used along the proofs contained in the paper. They are basically versions  for the dipolar kernel $K$ of some identities found by Cipolatti, see \cite{Cipo} in the context of the Davey-Stewartson system. For our equation \eqref{GP} the crucial identity is 
\begin{equation}\label{eq:der-ker}
\xi\cdot \nabla_\xi\hat K=0.
\end{equation}
We begin with the following.
\begin{lemma}
Given a (smooth) functions $f$ the following identity holds true:
\begin{equation}\label{eq:app-1}
2\int x\cdot\nabla\left(K\ast f\right)f\,dx=-3\int\left(K\ast f\right)f\,dx
\end{equation}
\end{lemma}
\begin{proof} The identity is proved by means of some properties of the dipolar kernel and its explicit representation in the frequencies space. We start by considering two (smooth) functions $f,g$ and we write 
\begin{equation*}
\begin{aligned}
\int x\cdot\nabla\left(K\ast f\right)g\,dx&=\int x\cdot\left(\nabla K\ast f\right)g\,dx=\int x \cdot\left(\int\nabla K(x-y)f(y)\,dy\right) g(x)\,dx\\
&=\iint (x-y)\cdot\nabla K(x-y)f(y)g(x)\,dy\,dx+\int \left(\int y\cdot\nabla K(x-y)f(y)\,dy\right)g(x)\,dx,
\end{aligned}
\end{equation*}
and since the dipolar kernel is an odd function, we continue as follows:
\begin{equation*}
\begin{aligned}
\int x\cdot\nabla\left(K\ast f\right)g\,dx&=\int \left( (x\cdot \nabla K) \ast f\right) g-\iint y\cdot\nabla K(y-x)f(y)g(x)\,dy\,dx \\
&=\int \left( (x\cdot \nabla K) \ast f\right) g\,dx-\iint x\cdot\nabla K(x-y)f(x)g(y)\,dx\,dy\\
&=\int \left( (x\cdot \nabla K) \ast f\right) g\,dx-\int x\cdot (\nabla K \ast g)f\,dx
\end{aligned}
\end{equation*}
hence, by using the Plancherel identity and \eqref{eq:der-ker} we conclude that 
\begin{equation}\label{eq:app-4}
\begin{aligned}
\int x\cdot\nabla\left(K\ast f\right)g\,dx+\int x\cdot\nabla\left(K\ast g\right)f\,dx&=\int \left( (x\cdot \nabla K) \ast f\right) g\,dx\\
&=-\int div_{\xi}(\xi\hat K)\hat f\bar{\hat g}=-3\int (K\ast f) g\,dx.
\end{aligned}
\end{equation}
By plugging $f=g$ in \eqref{eq:app-4} we prove \eqref{eq:app-1}.
\end{proof}
A consequence of the previous identities is the following.
\begin{lemma}
Given two (smooth) functions $f,g$ the following identity holds true:
\begin{equation}\label{eq:app-2}
\int x_3\partial_{x_3}\left(K\ast f\right)g\,dx+\int x_3\partial_{x_3}\left(K\ast g\right)f\,dx=-\int\left(K\ast f\right)g\,dx-\int\xi_3(\partial_{\xi_3}\hat K)\hat f\bar{\hat g}\,d\xi.
\end{equation}
\end{lemma}
\begin{proof}
Similarly to the previous proof we have:
\[
\begin{aligned}
\int x_3\partial_{x_3}\left(K\ast f\right)g\,dx+\int x_3\partial_{x_3}\left(K\ast g\right)f\,dx&=\int\left((x_3\partial_{x_3}K)\ast f\right)g\,dx=-\int\partial_{\xi_3}(\xi_3\hat K)\hat f\bar{\hat g}\,d\xi\\
&=-\int\left(K\ast f\right)g\,dx-\int\xi_3(\partial_{\xi_3}\hat K)\hat f\bar{\hat g}\,d\xi.
\end{aligned}
\]
\end{proof}

\section{A link between $\mathcal R^4_i$ and the parabolic biharmonic equation}\label{app:B}
In this section we provide a representation of the pairing between  a function $g\in L^2$ and the fourth power of the Riesz transform acting on a function $f\in L^2$. Let us consider the parabolic biharmonic equation
\begin{equation}\label{eq:biha}
\partial_t w+\Delta^2w=0, \qquad (t,x)\in\R^+\times\R^3.
\end{equation}
We denote by $P_t$ the linear propagator associated to \eqref{eq:biha}, namely $w(t,x):=P_tw_0(x)$ denotes the solution to the equation \eqref{eq:biha} with initial datum $w_0.$ We have the following representation result. 
\begin{prop}\label{biharm}
For any two functions $f,g \in L^2$, the following identity holds true:
\begin{equation}\label{eq:r4}
\langle \mathcal R_i^4f,g\rangle=-\int_0^\infty\langle \partial_{x_i}^4\frac{d}{dt}P_tf,g\rangle t\,dt.
\end{equation}
\end{prop}
\begin{proof}
By passing to the frequencies space, it is easy to see that $\widehat {P_tf}(\xi):=e^{-t|\xi|^4}\hat f(\xi)$ and we observe, by integration by parts, that 
\begin{equation}\label{eq:id4}
\xi^4_i|\xi|^4\int_0^\infty e^{-t|\xi|^4}t\,dt=\frac{\xi^4_i}{|\xi|^4};
\end{equation}
hence
\[
\begin{aligned}
\int_0^\infty\langle \partial_{x_i}^4\frac{d}{dt}P_tf,g\rangle t\,dt&=\langle\int_0^\infty \partial_{x_i}^4\frac{d}{dt}(P_tf)t\,dt, g \rangle=\langle\int_0^\infty \xi_i^4\frac{d}{dt}(e^{-t|\xi|^4}\hat f)t\,dt, \hat g\rangle\\
&=- \langle \xi_i^4|\xi|^4\hat f\int_0^\infty e^{-t|\xi|^4} t\,dt,\hat g\rangle= -\langle \frac{\xi_i^4}{|\xi|^4}\hat f,\hat g\rangle=-\langle \mathcal R_i^4f,g\rangle,
\end{aligned}
\]
where the change of order of integration (in time and in space) is justified by means of the Fubini-Tonelli's theorem, and we used the Plancherel identity when passing from the frequencies space to the physical space, and vice versa. 
\end{proof}
We  now observe that we can explicitly write the heat kernel associated to $P_t.$ We introduce, for $t>0$ and $x\in\R^3$ 
\[
p_t(x)=\alpha\frac{k(\mu)}{t^{3/4}}, \quad \mu=\frac{|x|}{t^{1/4}}
\]
and
\[
k(\mu)=\mu^{-2}\int_0^\infty e^{-s^4}(\mu s)^{3/2}J_{1/2}(\mu s)\,ds,
\]
where $J_{1/2}$ is the $\frac12$-th Bessel function, and $\alpha^{-1}:=\frac{4\pi}{3}\int_0^\infty s^2k(s)\,ds$ is a positive normalization constant. We refer to  \cite{FGG} for these definitions and further discussions about the heat kernel of  the parabolic biharmonic equation. 
We recall that the $\frac12$-th Bessel function is given by 
\[J_{1/2}(s)=(\pi/2)^{-1/2} s^{-1/2}\sin(s),
\] 
then  
\[
P_t f(x)=(p_t\ast f)(x)=c\int_{\R^3} f(x-y)\int_0^\infty\frac{1}{|y|^3}e^{-ts^4/|y|^4} s\sin{(s)}\,ds\,dy,
\]
and therefore 
\begin{equation}\label{eq:der:Pt}
\frac{d}{dt}P_t f(x)=-c\int_{\R^3}f(x-y)\int_0^\infty\frac{1}{|y|^3}e^{-ts^4/|y|^4} \frac{s^5}{|y|^4}\sin{(s)} \,ds\,dy.
\end{equation}
By combining \eqref{eq:r4} and \eqref{eq:der:Pt}, we can provide an alternative proof of the integral decay estimate similar to the one in \eqref{eq:int-decay-general}, of the following (weaker) type: for any couple of functions $f,g\in L^1\cap L^2,$ 
\begin{equation*}
|\langle \mathcal R_i^4f,g\rangle|\lesssim d^{-1}\|g\|_{L^1}\|f\|_{L^1},
\end{equation*}
where $d=\mbox{dist}(\mbox{supp}(f),\mbox{supp}(g)).$ We omit the details.

\section*{Acknowledgements}
\noindent 
\noindent J. B. is partially  supported  by Project 2016 ``Dinamica di equazioni nonlineari dispersive'' from  FONDAZIONE DI SARDEGNA.


\end{document}